\documentclass{amsart}

\usepackage{amssymb}
\usepackage[all]{xy}

\def\comment#1{{\sf{[#1]}}}
\def\Z{{\mathbb Z}}
\def\Q{{\mathbb Q}}
\def\C{{\mathbb C}}
\def\P{{\mathbb P}}

\def\F{{\mathbb F}}
\def\H{{\mathbb H}}
\def\L{{\mathbb L}}

\def\A{{\mathcal A}}

\def\cC{{\mathcal C}}

\def\cF{{\mathcal F}}
\def\cG{{\mathcal G}}

\def\M{{\mathcal M}}

\def\cP{{\mathcal P}}

\def\T{{\mathcal T}}
\def\U{{\mathcal U}}
\def\X{{\mathcal X}}

\def\G{\Gamma}

\def\d{{\mathfrak d}}
\def\g{{\mathfrak g}}

\def\n{{\mathfrak n}}
\def\p{{\mathfrak p}}
\def\r{{\mathfrak r}}

\def\u{{\mathfrak u}}

\def\etabar{{\overline{\eta}}}

\def\Ql{{\Q_\ell}}
\def\Zl{{\Z_\ell}}

\def\Gm{{\mathbb{G}_m}}
\def\Sp{{\mathrm{Sp}}}

\def\GSp{{\mathrm{GSp}}}

\def\prol{{(\ell)}}
\def\un{\mathrm{un}}

\def\arith{\mathrm{arith}}

\def\geom{\mathrm{geom}}

\def\cts{\mathrm{cts}}
\def\ur{\mathrm{ur}}
\def\rel{\mathrm{rel}}

\def\top{\mathrm{top}}
\def\orb{\mathrm{orb}}
\def\an{\mathrm{an}}
\def\ab{\mathrm{ab}}

\def\et{\mathrm{\acute{e}t}}

\def\sep{\mathrm{sep}}
\def\ram{\mathrm{ram}}

\def\Het{H_\et}

\def\dot{\bullet}

\def\bs{\backslash}

\newcommand\ad{\operatorname{ad}}

\newcommand\Hom{\operatorname{Hom}}

\newcommand\Spec{\operatorname{Spec}}

\newcommand\Aut{\operatorname{Aut}}
\newcommand\Inn{\operatorname{Inn}}
\newcommand\Out{\operatorname{Out}}
\newcommand\Der{\operatorname{Der}}

\newcommand\Gr{\operatorname{Gr}}

\newcommand\Gal{\operatorname{Gal}}

\newtheorem{theorem}{Theorem}[section]
\newtheorem{lemma}[theorem]{Lemma}
\newtheorem{proposition}[theorem]{Proposition}
\newtheorem{corollary}[theorem]{Corollary}
\newtheorem{bigtheorem}{Theorem}
\newtheorem{bigcorollary}[bigtheorem]{Corollary}

\theoremstyle{definition}
\newtheorem{definition}[theorem]{Definition}

\theoremstyle{remark}
\newtheorem{remark}[theorem]{Remark}

\newtheorem{variant}[theorem]{Variant}

\begin{document}
\title[Rational points of universal curves]{Rational points of universal curves in positive characteristics}

\author{Tatsunari Watanabe}
\maketitle
\begin{abstract}
Let $K$ be the function field of the moduli stack $\M_{g,n/\F_p}$ of curves over $\Spec\F_p$  and let $C/K$ be the restriction of the universal curve to $\Spec K$. We show that if $g\geq3$, then the only $K$-rational points of $C$ are its $n$ tautological points. Furthermore, we show that if $g\geq4$ and $n=0$, then Grothendieck's Section Conjecture holds for $C$ over $K$. This extends Hain's work in characteristic zero to positive characteristics. 
\end{abstract}
\tableofcontents

\section{Introduction} 
Suppose that $C$ is a geometrically irreducible smooth projective curve over a field $k$. Let $G_k$ be the absolute Galois group of $k$. Associated to the curve $C$, there is a short exact sequence of algebraic fundamental groups:
$$1\to \pi_1(C_{\bar k}, \bar x)\to \pi_1(C, \bar x)\to G_k\to 1,$$
where $\bar k$ is the separable closure of $k$ and $C_{\bar k}=C\otimes_k\bar k$. 
Each $k$-rational point $x$ of $C$ induces a section $s_x$ of $\pi_1(C, \bar x)\to G_k$, which is unique up to conjugation by elements of the geometric fundamental group $\pi_1(C_{\bar k}, \bar x)$. Grothendieck's section conjecture states that when $C$ is hyperbolic and $k$ is a finitely generated infinite field, there is a bijection between the set of $k$-rational points and the set of conjugacy classes of sections of $\pi_1(C, \bar x)\to G_k$ via the association 
$x\mapsto [s_x]$. Hain proved in \cite{hain2} that if  $g\geq 5$, char$(k)=0$, and the image of the $\ell$-adic cyclotomic character $\chi_\ell:G_k\to \Gm$ is infinity, the section conjecture holds for the restriction of the universal curve $\cC\to \M_{g/k}$ to its generic point $\Spec k(\M_g)$. In this paper, we will extend his results to positive characteristic.
 Before stating our main results, we need to introduce notations. A curve $C/T$ of type $(g,n)$ is a proper smooth morphism $C\to T$ whose geometric fibers are connected dimension-one schemes of arithmetic genus $g$ and that admits distinct $n$ sections $s_i: T\to C$. Suppose that $2g-2+n>0$. Let $k$ be a field.  Denote the moduli stack of curves of type $(g,n)$ over $\Spec(k)$ by $\M_{g,n/k}$ and the universal curve over it by $\cC_{g,n/k}$. Let $K$ be the function field of $\M_{g,n/k}$. The generic curve of type $(g,n)$ over $K$ with $g\geq3$ is the pullback of the universal curve $\cC_{g,n/k}$ to the function field $K$. The key ingredient that allows us to use Hain's methods in positive characteristic is the comparison of algebraic fundamental groups of a certain finite \'etale cover of $\M_{g,n}$. For a prime number   $\ell$, there is  a finite \'etale Galois cover $M^{\lambda}_{g,n}$ of  $\M_{g,n/\Z[1/\ell]}:=\M_{g,n/\Z}\otimes\Spec(\Z[1/\ell])$ that is representable by a scheme and has a smooth compactification over $\Z[1/\ell]$  whose boundary is a relative normal crossing divisor over $\Z[1/\ell]$.  Such covers were explicitly constructed by Boggi, de Jong, and Pikaart in \cite{BP},  \cite{deJPik}, and \cite{Pik}. Denote the moduli stack of curves of type $(g,n)$ over $\Spec(k)$ with an abelian level $r$ by $\M_{g,n/k}[r]$. When the ground field $k$ contains an $r$th root of unity $\mu_r(\bar k)$, we always assume that $\M_{g,n/k}[r]$ is a geometrically connected smooth stack over $\Spec(k)$. Suppose that $p$ is a prime number, $\ell$ is a prime number distinct from $p$, and $m$ is a nonnegative integer. Let $\cC_{g,n/\bar\F_p}[\ell^m]\to \M_{g,n/\bar\F_p}[\ell^m]$  be the universal curve over the stack $\M_{g,n/\bar\F_p}[\ell^m]$.
\begin{bigtheorem}Let $K$ be the function field of $\M_{g,n/\bar\F_p}[\ell^m]$. 
If $g\geq 4$ or if $g=3$, $p\geq 3$\footnote{The result is true for the case $g=3$ and $p=2$, and it is dealt with in the author's thesis.}, and $\ell=2$, then the only $K$-rational points of $\cC_{g,n/\bar \F_p}[\ell^m]$ are its $n$ tautological points. 

\end{bigtheorem}
The corresponding result in characteristic $0$ follows from results in Teichm$\ddot{\text{u}}$ller theory \cite{Hu,EaKr} due to Hubbard, Earle and Kra. Our approach is to apply Hain's algebraic methods in positive characteristics. \\
\indent Let $\F_q=\F_p[\zeta_{\ell^m}]$, where $\zeta_{\ell^m}$ is a primitive $\ell^m$th root of unity. 
\begin{bigtheorem} Let $C/L$ be the restriction of the universal curve\\ $\cC_{g/\F_q}[\ell^m]\to \M_{g/\F_q}[\ell^m]$ to the generic point $\Spec L$ of $\M_{g/\F_q}[\ell^m]$. Let $\bar L$ be the separable closure of $L$, and let $\bar x$ be a geometric point of $C_{\bar L}$. If $g\geq 4$, then the sequence
$$1\to\pi_1(C_{\bar L}, \bar x)\to \pi_1(C, \bar x)\to G_L\to1$$
does not split. 
\end{bigtheorem}
\begin{bigcorollary}
The section conjecture holds for the generic curve $C/L$. 
\end{bigcorollary}
\indent The first key tool used in this paper is the theory of  specialization homomorphism from \cite[SGA 1, \S X, XIII]{sga1}. This allows us to compare the maximal pro-$\ell$ quotient of the fundamental groups of $M^\lambda_{g,n/\bar\Q_p}$ and $M^\lambda_{g,n/\bar\F_p}$ when $\ell\not=p$. The essential tools used in Hain's original paper \cite{hain2} and this paper are weighted completion and relative completion of profinite groups. The theory of weighted completion was developed by Hain and Matsumoto in \cite{wei}. For a curve $C/T$, let $\GSp(H_{\Ql}):=\GSp(H^1_\et(C_{\etabar}, \Ql(1)))$ with $\ell$ a prime not in the residue characteristics char$(T)$ of $T$. There are natural monodromy actions of $\pi_1(C, \bar\eta)$ and $\pi_1(T, \bar\eta)$  into $\GSp(H_{\Ql})$ with the Zariski closure $R$ of their common images. One can take the weighted completion of $\pi_1(C, \bar \eta_C)$ and $\pi_1(T, \bar \eta_T)$ with respect to $R$ to obtain $\Ql$-proalgebraic groups $\cG_C$ and $\cG_T$. These are extensions of $R$ by  a prounipotent $\Ql$-group.  In this paper, $R$ is equal to the whole group $\GSp(H_{\Ql})$. For the universal curve $\cC_{g,n/k}\to\M_{g,n/k}$, the Zariski closure $\cG^\geom_{\M_{g,n/\bar k}}$ of the image in $\cG_{\M_{g,n/k}}(\Ql)$ of the composite $\pi_1(\M_{g,n/\bar k},\bar\eta)\to\pi_1(\M_{g,n/k},\bar\eta)\to 
\cG_{\M_{g,n/k}}(\Ql)$ is an extension of the reductive group $\Sp(H_{\Ql})$ by a prounipotent $\Ql$-group and its Lie algebra $\g^\geom_{g,n}$ is a pro-object of the category of the $\cG_{\M_{g,n/k}}$-modules. Each finite-dimensional $\cG_{\M_{g,n/k}}$-module $V$ admits a natural weight filtration: 
$$V=W_mV\supset W_{m-1}V\supset \cdots\supset W_n V$$
such that each weight graded quotient $\Gr^W_rV$ is a $\GSp(H_{\Ql})$-module of weight $r$. Each natural weight filtration induced on $\g^\geom_{g,n}$ satisfies the property that
$\g^\geom_{g,n}=W_0\g^\geom_{g,n}$ and its pronilpotent radical $\u^\geom_{g,n}$ is negatively weighted: 
$\u^\geom_{g,n}=W_{-1}\u^\geom_{g,n}$.  Theorem 1 and 2 are proved by using the structure of the truncated Lie algebra $\Gr^W_\bullet (\u^\geom_{g,n}/W_{-3})$.

{\em Acknowledgments:} I am truly grateful to my advisor Richard Hain for his support and many helpful, useful discussions. I am also very grateful to Makoto Matsumoto for his suggestions and comments on key technical parts of this work. 
\section{Fundamental Groups}
For a connected scheme $X$ and a choice of a geometric point $\bar\eta:\Spec \Omega \to X$, we have the \'etale fundamental group of $X$ denoted by $\pi_1(X, \bar\eta)$. More generally, for a Galois category $\cC$ with a fundamental functor $F$, we have the fundamental group $\pi_1(\cC, F)$. When $\cC$ is the category of finite \'etale covers $E$ of $X$ and $F=F_{\bar\eta}: E \mapsto E_{\bar\eta}:=E\times_X\Spec \Omega$, we have $\pi_1(\cC, F)=\pi_1(X, \bar\eta)$. When $X$ is a field $k$  and $\bar k$ is an algebraic closure of $k$, we have $\pi_1(\Spec k, \Spec \bar k)= G_k:=\Gal(k_{\sep}/k)$, where $k_{\text{sep}}$ is the separable closure of $k$ in $\bar k$. 
In this paper, we will need the extension of this theory to the Deligne-Mumford stacks, which are constructed in \cite{noo}. 
\subsection{Comparison theorem} Suppose that $k$ is a subfield of $\C$. Let $\bar k$ be the algebraic closure of $k$ in $\C$. For a geometrically connected scheme $X$ of finite type over $k$ and a geometric point $\bar\eta: \Spec \C\to X$,  there is a canonical isomorphism 
$$ \pi_1^\top(X^\an, \bar\eta)^\wedge\cong \pi_1(X\otimes_k\bar k, \bar \eta), $$
where $X^\an$ denotes the complex analytic variety associated to $X$ and $\pi_1^\top(X^\an, \bar\eta)^\wedge$ denotes the profinite completion of the topological fundamental group of $X^\an$ with the image of $\bar\eta$ as  a base point.  Furthermore, for a DM stack $\X$ over $k$, the corresponding analytical space denoted by $\X^\an$ is an orbifold (or a stack in the category of topological spaces) and we have the orbifold fundamental group $\pi_1^\orb(\X^\an, x)$ of $\X^\an$ with an appropriate base point $x\to \X^\an$. The above comparison theorem extends to DM stacks over $k$ (see \cite{noo2} for details): there is a canonical isomorphism 
$$\pi_1^\orb(\X^\an, x)^\wedge\cong \pi_1(\X\otimes_k\bar k,x ),$$
where $x:\Spec \C\to\X$ is a geometric point of $\X$. 
\subsection{Fundamental groups of curves}
Let $C$ be a smooth curve of genus $g$ over an algebraically closed field $k$ such that $C$ is a complement of $n\geq 0$ closed points of its smooth compactification. Fix a geometric point $\bar\eta$ of $C$. When char$(k)=0$, the fundamental group of a smooth curve does not change under extensions of algebraically closed fields of characteristic zero \cite[5.6.7]{sza}, and thus we may assume that $k$ is a subfield of $\C$. Then by the comparison theorem the fundamental group $\pi_1(C, \bar\eta)$ of $C$ with base point $\bar\eta$ is isomorphic to the profinite completion of the group
$$\Pi_{g,n}:=\langle a_1, b_1, \ldots, a_g, b_g,\gamma_1,\ldots,\gamma_n|[a_1,b_1][a_2,b_2]\cdots[a_g,b_g]\gamma_1\cdots\gamma_n=1\rangle.$$
When char$(k)=p>0$, Grothedieck proved in \cite{sga1} that the maximal prime-to-$p$ quotient\footnote{Here the maximal prime-to-$p$ quotient $G^{(p')}$ of a profinite group $G$ is the projective limit of its finite continuous quotients of order prime to $p$. 
}of $\pi_1(C, \bar\eta)$, denoted by $\pi_1(C, \bar\eta)^{(p')}$, is isomorphic to the maximal prime-to-$p$ completion of the group $\Pi_{g,n}$.  
\subsection{Fundamental group of the generic point of a variety}
Suppose that $X$ is a smooth variety over a field $k$. Let $K=k(X)$ be the function field of $X$ and $\bar\eta$ be a geometric point lying over the generic point of $X$. We may take this geometric point $\bar\eta$ as a base point for any open subvariety of $X$. For divisors $D\subset E$ of $X$ defined over $k$, there is a canonical surjection
$$\pi_1(X-E,\bar\eta)\to\pi_1(X-D,\bar\eta)$$
and thus there is a projective system of profinite groups:
$$\{\pi_1(X-D, \bar\eta)\}_D,$$
where $D$ is taken over the divisors of $X$ defined over $k$. 
Fix an algebraic closure $\bar K$ of $K$. Let $K_{\sep}$ be the separable closure of $K$ in $\bar K$. Then the Zariski-Nagata \cite[Theorem 3.1]{sga1} implies 
\begin{proposition}
The canonical surjection
$$G_K\to \varprojlim_D\pi_1(X-D, \bar\eta)$$
 is an isomorphism.
\end{proposition}

\section{Monodromy Representation}
Suppose that $S$ is a connected scheme, and that $f:X\to S$ is a proper smooth morphism of schemes whose fibers are geometrically connected. Let $\bar s: \Spec \Omega\to S$ be a geometric point of $S$ and $\bar x$ be a geometric point of the fiber $X_{\bar s}$ of $X$ with a value in $\Omega$. Let char $(S)$ be the set of residue characteristics of $S$ and let $\mathbb{L}$ be a set of prime numbers not in char$(S)$.  The following results are from \cite[SGA 1, Expos$\acute{\text{e}}$ XIII, 4.3, 4.4]{sga1}. Let $K$ be the kernel of the canonical homomorphism $\pi_1(X,\bar x)\to\pi_1(S, \bar s)$ and $N$ be the kernel of the projection $K\to K^{\L}$ where $K^{\L}$ is the maximal pro-$\L$ quotient of $K$. Then $N$ is a distinguished subgroup of $\pi_1(X,\bar x)$ and we denote by $\pi_1'(X, \bar x)$ the quotient of $\pi_1(X, \bar x)$ by $N$. Also we denote by $\pi_1^{\L}(X_{\bar s}, \bar x)$ the maximal pro-$\L$ quotient of $\pi_1(X_{\bar s}, \bar x)$.   In general, the sequence 
$$ \pi_1^{\L}(X_{\bar s}, \bar x)\to \pi_1'(X,\bar x)\to\pi_1(S, \bar s)\to 1$$
is exact, but if the morphism $f: X\to S$ admits a section, it becomes also left exact:
$$1\to\pi_1^{\L}(X_{\bar s}, \bar x)\to \pi_1'(X, \bar x)\to \pi_1(S, \bar s)\to 1.$$
  In this case, we obtain a monodromy action
$$\rho_{\bar s}: \pi_1(S, \bar s)\to \text{Out}(\pi_1^{\L}(X_{\bar s}, \bar x)).$$
\begin{variant}When the set $\L$ contains only one prime number $\ell$, we denote $\pi_1^{\L}(X_{\bar s}, \bar x)$ by $\pi_1^{(\ell)}(X_{\bar s}, \bar x)$ instead.
\end{variant}
\begin{proposition} Suppose that $T$ is a locally noetherian, connected scheme, and that $\ell$ is not in char$(T)$. Let $f:C\to T$ be a curve of genus $g\geq 2$.  Then the sequence
$$1\to\pi_1^{(\ell)}(C_{\bar\eta}, \bar x)\to \pi_1'(C, \bar x)\to \pi_1(T, \bar\eta)\to 1$$
is exact.
\end{proposition}
\begin{proof}
Suppose $g\geq2$. Let $\cC_g\to\M_g$ be the universal curve of genus $g$. By assumption, we have the fiber product
$$
\xymatrix{C\ar[r]\ar[d]&\cC_{g/\Z[1/\ell]}\ar[d]\\
          T\ar[r]    &\M_{g/\Z[1/\ell]}          
}
$$
where $\M_{g/\Z[1/\ell]}:=\M_g\otimes_{\Z}\Z[1/\ell]$. Thus it will suffice to show for the universal curve $\cC_{g/\Z[1/\ell]}\to\M_{g/\Z[1/\ell]}$. This follows from the commutative 
diagram
$$
\xymatrix@R=10pt@C=10pt{&\pi_1^{(\ell)}(C_\etabar)\ar[r]\ar@{=}[d]&\pi_1'(\cC_{g/\Z[1/\ell]})\ar[r]\ar[d]&\pi_1(\M_{g/\Z[1/\ell]})\ar[r]\ar[d]&1\\
1\ar[r]&\pi_1^{(\ell)}(C_\etabar)\ar[r]&\G^{\arith,(\ell)}_{g,1}\ar[r]&\G^{\arith,(\ell)}_g\ar[r]&1,
}
$$
where the profinite group $\G^{\arith,(\ell)}_{g,n}$ is the fundamental group of the Galois category $\cC(\M_{g,n/\Z[1/\ell]})$ defined in \cite[\S 7]{rel} and the rows are exact.
\end{proof}
\indent Suppose that $T$ is a locally noetherian connected scheme,  and that $C\to T$ is a curve. Fix a prime number $\ell$ not in char$(T)$.
 Then we have the exact sequence
$$1\to\pi_1(C_{\bar\eta}, \bar x)^{(\ell)}\to \pi_1'(C, \bar x)\to \pi_1(T, \bar\eta)\to 1,$$
from which we obtain a natural monodromy action of $\pi_1(T, \bar\eta)$ on \\
$\Hom(\pi_1(C_{\bar\eta}, \bar x)^{(\ell)}, \Zl(1))\cong H_{et}^1(C_{\bar\eta}, \Zl(1))$. Denote $H_{et}^1(C_{\bar\eta}, \Zl(1))$ by $H_{\Zl}$. The action of $\pi_1(T, \bar\eta)$ respects the Weil pairing $\theta: \Lambda ^2H_{\Zl}\to \Zl(1)$. Hence we obtain a representation  
$$\rho_{\bar\eta}: \pi_1(T, \bar\eta)\to \GSp(H_{\Zl}).$$
In particular, when $T$ is defined over a field $k$, we have the commutative diagram
$$\xymatrix{
 \pi_1(T,\bar{\eta})\ar[r]^{\rho_{\bar{\eta}}}\ar[d]&\GSp(H_{\Zl})\ar[d]^\tau\\
 G_k\ar[r]^{\chi_{\ell}}&\Gm(\Zl)}
$$
 where the left-hand vertical map is the canonical projection, the right-hand vertical map $\tau$ is the natural surjection, and where $\chi_{\ell}$ is the $\ell$-adic cyclotomic character. 

\section{Moduli of Curves with a Teichm$\ddot{\text{u}}$ller Level Structure}

\subsection{Moduli stacks of curves with a non-abelian level structure}
 Suppose that $2g-2+n>0$. Denote the Deligne-Mumford compactification \cite{DM} of $\M_{g,n/\Z}$ by $\overline{\M}_{g,n/\Z}$.   Fix a prime number $\ell$. Finite \'etale coverings of $\M_{g,n}$ that are representable by a scheme and have a compactification that is smooth over $\Spec \Z[1/\ell]$ are essential to our comparison between characteristic zero and positive characteristic.  The existence of such coverings was established by 
 \begin{enumerate}
 \item de Jong and Pikaart for $n=0$ and all $\ell$ in \cite{deJPik},
 \item Boggi and Pikaart for $n>0$ and odd $\ell$ in \cite{BP}, and
 \item Pikaart for $n>0$ and $\ell=2$ in \cite{Pik}.
 \end{enumerate}
 Their results needed in this paper are summarized in the following statement:
 
 \begin{proposition}\label{prop. 4}
 For all prime numbers $\ell$ and all $(g,n)$ satisfying $2g-2+n>0$, there is a finite \'etale Galois covering $M\to \M_{g,n}[1/\ell]:=\M_{g,n/\Z}\otimes\Z[1/\ell]$ 
over $ \Z[1/\ell]$ that satisfies:
\begin{enumerate}
 \item $M$ is a separated scheme of finite type over $\Z[1/\ell]$;
 \item the normalization $\overline{M}$ of $\overline{\M}_{g,n}[1/\ell]$ with respect to $M$ is proper and smooth over $\Z[1/\ell]$;
 \item the boundary $\overline{M}\bs M$ is a relative normal crossing divisor over $\Z[1/\ell]$.
\end{enumerate}
\end{proposition}
In fact, $M$ was taken to be the DM stack ${}_{G}\M_{g,n/\Z[1/\ell]}$ of curves of type $(g,n)$ with a
Teichm$\ddot{\text{u}}$ller structure of level $G$ (see \cite{DM} for definition), where $G$ was specifically taken to be:
\begin{enumerate}
 \item the quotient of $\Pi_{g,0}$ by the normal subgroup generated by the third term of its lower central subgroup and all $\ell^m$th powers
when $\ell$ is odd and $n=0$;
\item the quotient of $\Pi_{g,0}$ by the normal subgroup generated by the fourth term of its lower central subgroup and all fourth powers
when $\ell=2$  and $n=0$;
\item the quotient $\Pi_{g,n}/W^3\Pi_{g,n}\cdot\Pi_{g,n}^{\ell^m}$, where $W^3$ denotes
 the third term of the weight filtration of $\Pi_{g,n}$ defined in \cite{BP} when $\ell$ is odd and $n>0$;
\item the quotient $\Pi_{g,n}/W^4\Pi_{g,n}\cdot\Pi_{g,n}^4$, where $W^4$ denotes
 the fourth term of the weight filtration of $\Pi_{g,n}$ defined in \cite{BP} when $\ell=2$ and $n>0$,
\end{enumerate}
where $\Pi^k_{g,n}$ is the subgroup of $\Pi_{g,n}$ generated by all $k$th powers. In \cite{DM}, $G$ is a finite quotient of $\Pi_{g,n}$ by a characteristic subgroup, but the same construction can be done when $G$ is a finite quotient of $\Pi_{g,n}$ by an invariant subgroup, see \S 5.4.  For $n\geq 2$, the subgroups $W^\bullet\Pi_{g,n}\cdot\Pi_{g,n}^{k}$ are not characteristic, but are invariant. For fixed prime numbers $p$ and $\ell\not=p$, denote by $M^\lambda_{g,n}$ or simply $M^\lambda$ the finite \'etale cover $M$ of $\M_{g,n}[1/\ell]$ given by  the above proposition. 
\subsection{Moduli stacks of curves with an abelian level} When $G$ is a finite quotient by the subgroup  $W^2\Pi_{g,n}\cdot\Pi_{g,n}^{m}$, we have $G\cong H_1(\Sigma_g, \Z/m\Z)$, where $\Sigma_g$ is a closed oriented genus $g$ surface.  In this case, we denote the moduli stack of $n$-pointed smooth projective curves with the Teichm$\ddot{\text{u}}$ller structure of level $H_1(\Sigma_g, \Z/m/\Z)$ by $\M_{g,n}[m]$.  The stack $\M_{g,n}[m]$ is representable by a scheme for $m\geq 3$ (See \cite[Chapter XVI, Theorem 2.11]{goac}. It is well known that the Deligne-Mumford compactification $\overline{M_{g,n}[m]}$ is never smooth if $g>2$.

 \subsection{Fundamental Groups of Finite \'Etale Covers of Moduli Stacks of Curves}
 
Suppose that $g$ and $n$ are non-negative integers satisfying $2g-2+n>0$. Fix a closed oriented genus $g$ surface $\Sigma_g$ and a finite subset $P=\{p_1,p_2,\ldots,p_n\}$ of $n$ distinct points in $\Sigma_g$. Denote the mapping class group of $(\Sigma_g, P)$ by $\G_{\Sigma_g,P}$. This is defined to be the group of isotopy classes of orientation preserving homeomorphisms which fix $P$ pointwise. By the classification of surfaces, the homeomorphism class of $(\Sigma_g,P)$ depends only on $(g,n)$. Therefore, the group $\G_{\Sigma_g,P}$ depends only on the pair $(g,n)$, and thus it is denoted by $\G_{g,n}$. Denote the complement $\Sigma_g-P$ of $P$ in $\Sigma_g$ by $\Sigma_{g,n}$. Denote the topological fundamental group $\pi_1^\top(\Sigma_{g,n}, \ast)$ of $\Sigma_{g,n}$ by $\Pi_{g,n}$. The standard presentation of $\Pi_{g,n}$ is 
$$\Pi_{g,n}=\langle\alpha_1, \beta_1, \ldots, \alpha_g, \beta_g, \gamma_1,\ldots,\gamma_n|[\alpha_1, \beta_1]\cdots[\alpha_g, \beta_g]\gamma_1\cdots\gamma_n=1\rangle.$$
Note that $\Pi_{g,0}=\Pi_{g,n}/\langle \gamma_1, \ldots, \gamma_n\rangle$. 
The geometric automorphisms of $\Pi_{g,n}$ are defined to be the ones that fix the conjugacy class of every $\gamma_i$ and induce the identity  on $H_2(\Pi_{g,0},\Z)$. Denote the group of geometric automorphisms of $\Pi_{g,n}$ by $A_{g,n}$ and the group of the inner automorphisms of $\Pi_{g,n}$ by $I_{g,n}$. The group $I_{g,n}$ is clearly a normal subgroup of $A_{g,n}$. It is well known that there is a canonical isomorphism 
\[\G_{g,n}\cong A_{g,n}/I_{g,n}\]
(See \cite[Theorem V.9]{zvc}).
The invariant subgroups of $\Pi_{g,n}$ are defined to be the ones that  are stable under the action of $A_{g,n}$. For an invariant subgroup $K$ of $\Pi_{g,n}$, there is a natural representation 
\[\G_{g,n}\to \Out(\Pi_{g,n}/K).\] This representation is the key for the construction of $M^\lambda$. \\
\indent Let $k$ be a field of characteristic $0$. For simplicity, assume that $k$ is contained in $\C$ and denote the algebraic closure of $k$ in $\C$ by $\bar{k}$. The moduli stack  $\M_{g,n/\C}$ can be viewed as a complex analytic orbifold denoted by $\M^\an_{g,n/\C}$. Denote the orbifold fundamental group of $\M^\an_{g,n/\C}$ by $\pi_1^\orb(\M^\an_{g,n/\C}, \bar\eta)$ with base point $\bar\eta \in \M_{g,n}(\C)$. There is a natural isomorphism 
$$
\pi_1^\orb(\M_{g,n/\C},\bar\eta)\cong \G_{g,n}.
$$
Therefore, for each geometric point $\bar{\eta}$ of $\M_{g,n/\bar{k}}$, there is an isomorphism
$$
\pi_1(\M_{g,n/\bar{k}},\bar{\eta})\cong \G_{g,n}^\wedge,
$$
which is uniquely determined up to inner automorphisms, and there is an exact sequence
$$
1\to \G_{g,n}^\wedge\to\pi_1(\M_{g,n/k},\bar{\eta})\to G_k\to1.
$$

 Let $k$ be an algebraically closed field of characteristic $p>0$. Denote the ring of $p$-adic Witt vectors over $k$ by $W(k)$. When $k$ is clear from context, we denote $W(k)$ by $W$.  It is a characteristic zero complete discrete valuation ring with the residue field $k$. Fix an algebraic closure $L$ of the fraction field of $W(k)$. There is an isomorphism  $\G_{g,n}^\wedge\cong\pi_1(\M_{g,n/L},\bar{\eta})$ of the geometric fundamental group of $\M_{g,n/L}$ with the profinite completion of the mapping class group $\G_{g,n}$.  
Fix a prime number $\ell\not =p$. Let $G=\Pi_{g,n}/W^3\Pi_{g,n}\cdot\Pi_{g,n}^{\ell^m}$ for odd $\ell$ and $G=\Pi_{g,n}/W^4\Pi_{g,n}\cdot\Pi_{g,n}^4$ for $\ell=2$. Let $M^\lambda$ be a finite \'etale cover of $\M_{g,n}[1/\ell]$ as in Proposition \ref{prop. 4}.  Denote the kernel of the natural representation $\Gamma_{g,n}\rightarrow \text{Out}(G)$ by $\Gamma^{\lambda}_{g,n}$. Denote the Teichm$\ddot{\text{u}}$ller space of the reference surface $\Sigma_{g,n}$ by $\T_{g,n}$.  By construction, each connected component of the complex variety $M^\lambda\otimes \C$ is isomorphic to the analytic space $\T_{g,n}/\G^\lambda_{g,n}$. Since $\G^\lambda_{g,n}$ acts on $\T_{g,n}$ freely, we see that there is a natural conjugacy class of isomorphisms
$$
\pi_1(M^\lambda_{\C})\cong(\G^\lambda_{g,n})^\wedge,
$$ 
where $M^\lambda_{\C}$ is a connected component of $M^\lambda\otimes \C$. 
Since $\ell$ is a unit in $W$, there is a natural morphism $\Spec W\to\Spec\Z[1/\ell]$. Choose a connected component of $M^\lambda\otimes_{\Z[1/\ell]} W$ and denote it by $M^\lambda_{W}$. Denote its base changes  to $L$ and $k$ by $M^\lambda_{L}$ and $M^\lambda_k$, respectively. Let $\bar\eta$ and $\bar\xi$  be a geometric point of $M^\lambda_L$ and $M^\lambda_k$, respectively. The scheme $M^\lambda_{L}$ is a connected finite \'etale cover of $\M_{g,n/L}$ and there is an isomorphism $ \pi_1(M^{\lambda}_{L}, \bar{\eta})\cong(\G^\lambda_{g,n})^\wedge$. Since the boundary of $\overline{M^\lambda}$ is a relative normal crossing divisor over $\Z[1/\ell]$, the boundary of the Zariski closure of $M^\lambda_{W}$ in $\overline{M^\lambda}\otimes W$ is also a relative normal crossing divisor over $W$.  This allows us to define a specialization homomorphism of tame fundamental groups \cite[Expos\'e XIII]{sga1}
 $$sp: \pi_1^{t}(M^\lambda_{L}, \bar{\eta})\to \pi_1^t(M^\lambda_{W}, \bar{\eta})\cong\pi_1^t(M^\lambda_{W}, \bar{\xi})\overset{\sim}\leftarrow\pi_1^t(M^\lambda_{k}, \bar{\xi}),$$
 where the left-hand map is induced by base change to $L$, the map at middle is an isomorphism obtained by change of base points, and the right-hand map is the isomorphism induced by base change to $k$.  \begin{theorem}\label{thm 1}With notations as above, there is an isomorphism
 \[(\G^\lambda_{g,n})^\prol \cong \pi_1(M^\lambda_{k}, \bar{\xi})^\prol,\] 
 which is uniquely determined up to inner automorphisms. 
 \end{theorem}
 \begin{proof}
 The smoothness of $M^\lambda_{W}$ over $W$ implies that the specialization morphism $sp$ is surjective. This surjective homomorphism induces an isomorphism 
 $$sp^{(p')}: \pi_1(M^\lambda_{L}, \bar{\eta})^{(p')}\overset{\sim}\to \pi_1(M^\lambda_{k}, \bar{\xi})^{(p')}$$
 upon taking maximal prime-to-$p$ quotient \cite[Expos\'e XIII]{sga1}. Hence we have an isomorphism
 $$sp^{\prol}: \pi_1(M^\lambda_{L}, \bar{\eta})^{\prol}\overset{\sim}\to \pi_1(M^\lambda_{k}, \bar{\xi})^{\prol}$$
 by taking maximal pro-$\ell$ quotient. 

 \end{proof}
 \begin{corollary}
 With notations as above, there are natural conjugacy classes of isomorphisms
 $$(\G_{g,n}[\ell^m])^{(\ell)}\cong\pi_1(\M_{g,n/ k}[\ell^m])^{(\ell)}$$
 and
 $$\G_{g,n}^{\rel(\ell)}\cong\pi_1(\M_{g,n/ k})^{\rel(\ell)},$$
 where $\rel(\ell)$ denotes relative pro-$\ell$ completion with respect to the natural homomorphism to
 $\Sp_g(\Zl)$ (see \cite{rel} for definition and basic results).
 
 \end{corollary}
 \begin{proof}For $A=L, W$, and $k$, denote $\M_{g,n/A}$ and $\M_{g,n/A}[\ell^m]$ by $\M_A$ and $\M_A[\ell^m]$, respectively. 
 Let $\bar\eta$ and $\bar\xi$ be geometric points of $M^\lambda_L$ and $M^\lambda_k$, respectively. Denote the images of $\bar\eta$ and $\bar\xi$ under morphisms by $\bar\eta$ and $\bar\xi$ also. The monodromy action $\pi_1(\M_A)^{\rel(\ell)}\to \Sp(\Z/\ell\Z)$\footnote{For $\ell=2$, the same statement is true with $\Sp(\Z/4\Z)$.} factors through the finite group $\G_{g,n}/\G^\lambda_{g,n}$, which is the automorphism group of $M^\lambda_A$ over $\M_A$. Denote this finite group by $G$. This implies that for $A=W$ and $A=k$, there is an exact sequence
 $$1\to\pi_1(M^\lambda_A, \bar\xi)^{(\ell)}\to\pi_1(\M_A, \bar\xi)^{\rel(\ell)}\to G\to 1.$$
 Similarly, for $A=L$ and $A=W$, there is an exact sequence
 $$1\to\pi_1(M^\lambda_A, \bar\eta)^{(\ell)}\to\pi_1(\M_A, \bar\eta)^{\rel(\ell)}\to G\to 1.$$ 
 Fix an isomorphism $\pi_1(M^\lambda_W, \bar\xi)\cong\pi_1(M^\lambda_W, \bar\eta)$. These exact sequences fit into the commutative diagram
 $$
 \xymatrix@R=10pt@C=10pt{
 1\ar[r]&\pi_1(M^\lambda_k, \bar\xi)^{(\ell)}\ar[r]\ar[d]&\pi_1(\M_k, \bar\xi)^{\rel(\ell)}\ar[r]\ar[d]&G\ar[r]\ar@{=}[d]&1\\
 1\ar[r]&\pi_1(M^\lambda_W, \bar\xi)^{(\ell)}\ar[r]\ar[d]&\pi_1(\M_W, \bar\xi)^{\rel(\ell)}\ar[r]\ar[d]&G\ar[r]\ar[d]&1\\
 1\ar[r]&\pi_1(M^\lambda_W, \bar\eta)^{(\ell)}\ar[r]&\pi_1(\M_W, \bar\eta)^{\rel(\ell)}\ar[r]&G\ar[r]&1\\
 1\ar[r]&\pi_1(M^\lambda_L, \bar\eta)^{(\ell)}\ar[r]\ar[u]&\pi_1(\M_L, \bar\eta)^{\rel(\ell)}\ar[r]\ar[u]&G\ar[r]\ar@{=}[u]&1,
 }
 $$
 where the left-hand vertical maps are all isomorphisms and the map $G\to G$ is an isomorphism induced by the fixed isomorphism $\pi_1(M^\lambda_W, \bar\xi)\cong\pi_1(M^\lambda_W, \bar\eta)$.
  Therefore, the middle vertical maps are all isomorphisms and thus
  there are isomorphisms
  $$\pi_1(\M_k, \bar\xi)^{\rel(\ell)}\cong \pi_1(\M_L, \bar\eta)^{\rel(\ell)}\cong \G_{g,n}^{\rel(\ell)},$$
  which are unique up to conjugation by elements of $\pi_1(\M_k, \bar\xi)^{\rel(\ell)}$. Similarly, let $G'$ be the quotient of $\pi_1(\M_A[\ell^m])$ by the finite index subgroup $\pi_1(M^\lambda_A)$. It is a finite $\ell$-group. Using the exact sequences 
  $$1\to\pi_1(M^\lambda_A, \bar\xi)^{(\ell)}\to\pi_1(\M_A[\ell^m], \bar\xi)^{(\ell)}\to G'\to 1,$$
  where $A=W$ and $A=k$, 
   and 
 $$ 1\to\pi_1(M^\lambda_A, \bar\eta)^{(\ell)}\to\pi_1(\M_A[\ell^m], \bar\eta)^{(\ell)}\to G'\to 1,$$
 where $A=L$ and $W$, we also have isomorphisms
 $$\pi_1(\M_k[\ell^m], \bar\xi)^{(\ell)}\cong \pi_1(\M_L[\ell^m], \bar\eta)^{(\ell)}\cong \G_{g,n}[\ell^m]^{(\ell)},$$
  which are unique up to conjugation by elements of $\pi_1(\M_k[\ell^m], \bar\xi)^{(\ell)}$.
 
  \end{proof}

\section{Relative completion of $\G^\lambda_{g,n}$} Suppose $2g-2+n> 0$. Let $H_A=H_1(\Sigma_g, A)$, where $H_1(\Sigma_g,A)$
is the fisrt homology group of the compact reference surface $\Sigma_g$. Let $\rho:\Gamma_{g,n}\rightarrow \Sp(H_{\Q})$ be the representation of the mapping
class group on the first homology of the surface. Since the image of $\rho$ is $\Sp(H_{\Z})$, $\rho$ is a Zariski dense representation. Denote by $\cG^{\geom}_{g,n}$
the relative completion (see \cite{hain4} for definition and basic properties) of $\Gamma_{g,n}$ with respect to $\rho$ and by $\U^{\geom}_{g,n}$ its prounipotent radical. The relative completion behaves well under base change. For instance, we have that the relative completion of $\G_{g,n}$ with respect to $\rho:\G_{g,n}\to\Sp(H_{\Ql})$ is isomorphic to $\cG^\geom_{g,n}\otimes_\Q\Ql$.\\
 Let $\ell$ be an odd prime number. Recall that $\G^\lambda_{g,n}$ is the kernel of the natural representation $\Gamma_{g,n}\rightarrow \text{Out}(G)$, where $G=\Pi_{g,n}/W^3\Pi_{g,n}\cdot\Pi_{g,n}^{\ell^m}$. 
The following theorem follows from \cite[Cor.~6.7]{Kahler}.
\begin{theorem} \label{thm.1}Suppose that $g\geq 3$ and $n\geq0$.  The completion of $\G^\lambda_{g,n}$ relative to the restriction of the standard representation $\rho:\Gamma_{g,n}\to\Sp(H_{\Q})$ is 
isomorphic to $\cG^{\geom}_{g,n}$
\end{theorem}

Suppose that $\G$ is a profinite group and that $\rho:\G\to R(\Zl)$ is a continuous homomorphism such that the composition with the inclusion $R(\Zl)\to R(\Ql)$ has Zariski dense image. Let $\rho^{\rel(\ell)}:\G^{\rel(\ell), \rho}\to R(\Zl)$ be the relative pro-$\ell$ completion of $\G$ with respect to $\rho$ (see \cite{rel} for definition). Since $\G\to\G^{\rel(\ell),\rho}$ is surjective, $\rho^{\rel(\ell)}: \G^{\rel(\ell),\rho}\to R(\Ql)$ has Zariski dense image. The following result easily follows from the universal property of relative completion.

\begin{proposition}\label{prop.7}
The continuous relative completion of $\G^{\rel(\ell),\rho}$ with respect to the homomorphism $\rho^{\rel(\ell)}: \G^{\rel(\ell),\rho}\to R(\Ql)$ is isomorphic to the continuous relative completion $\cG$ of $\G$ with respect to $\rho$. \qed
\end{proposition}

\section{Weighted Completion and Families of Curves}
\subsection{Review of weighted completion of a profinite group} Weighted completion of a profinite group $\G$ is similar to continuos relative completion. It plays an essential role in \cite{hain2}. A key property of weighted completion is that it induces weight filtrations with strong exactness properties on the $\G$-representations that factor through its weighted completion.
Here we take $F$ to be $\Ql$, where $\ell$ is a prime number.
Denote $\Gm_{/\Ql}$ by $\Gm$. Suppose that:
\begin{enumerate}
 \item $\G$ is a profinite group;
 \item $R$ is a reductive algebraic group defined over $\Ql$;
 \item $w:\Gm\to R$ is a central cocharacter;
 \item $\rho:\G\to R(\Ql)$ is a continuous homomorphism with Zariski dense image.
\end{enumerate}
\begin{definition}[{\cite[\S4]{wei}}]
 The \textit{weighted completion} of $\G$ with respect to $\rho$ and $w$ consists of a proalgebraic $\Ql$-group $\cG$, 
that is a negatively weighted extension
\[1\to \U\to \cG\to R\to 1\]
where $\U$ is a prounipotent $\Ql$-group and a continuous Zariski dense homomorphism $\tilde{\rho}:\G\to\cG(\Ql)$ whose composition with $\cG(\Ql)\rightarrow R(\Ql)$ is $\rho$.
It is characterized by the following universal mapping property:
If $G$ is an affine (pro)algebraic $\Ql$-group that is a negatively weighted extension\footnote{Viewing $H_1(U)$ as a $\Gm$-module via $\omega$, it admits only negative weights:
$H_1(U)=\oplus_{n<0}H_1(U)_n$, where $\Gm$ acts on $H_1(U)_n$ via the $n$th power of the defining representation.}
\[1\to U\to G\to R\to 1\]
of $R$ (with respect to $w$) by a (pro)unipotent group $U$, and if $\phi:\G \to G(\Ql)$ is a continuous homomorphism  whose composition with $G(\Ql)\to R(\Ql)$ is $\rho$,
then there is a unique homomorphism of proalgebraic $\Ql$-groups $\Phi:\cG\to G$ that commutes with the projections to $R$ and such that $\phi=\Phi\circ\tilde{\rho}$:
\[
\xymatrix{\Gamma\ar[r]^{\tilde{\rho}}\ar[d]_{\phi}&\cG\ar[d]\ar[ld]_{\Phi}\\
G\ar[r]&R\\
}
\]
\end{definition}
\begin{proposition}\label{prop.3}{\cite[Thms.~3.9 \& 3.12]{wei}}
 Every finite dimensional $\cG$-module $V$ has
a natural weight filtration $W_\dot$:
\[0=W_nV\subset\cdots\subset W_{r-1}V\subset W_rV\subset\cdots W_mV=V.\]
It is characterized by the property that the action of $\cG$ on the $r$th weight graded quotient
\[\Gr^W_rV:=W_rV/W_{r-1}V\]
factors through $\cG\to R$ and is an $R$-module of weight $r$. The weight filtration is preserved by $\cG$-module homomorphisms and the 
functor $\Gr^W_\dot$ on the category of finite-dimensional $\cG$-modules is exact. 
\end{proposition}


Suppose that $V$ is a finite-dimensional $R$-representation. The representation $V$ can be decomposed as $V=\bigoplus_{\n\in \Z}V_n$ under the $\Gm$-action through $\omega$. We say that $V$ is {\it pure} of weight $n$ if $V=V_n$,  and that $V$ is {\it negatively weighted} if $V_n=0$ for all $n\geq 0$. $V$ can be considered as a continuous $\G$-module via the homomorphism $\rho:\G\to R(\Ql)$. Denote by $H^\bullet_\cts(\G, V)$ the continuous cohomology of $\G$ with coefficients in $V$. 

\begin{proposition}[{\cite[Thms. 4.6 \& 4.9]{wei}}]\label{prop. 16}
For all finite-dimensional irreducible $R$-representations $V$ of weight $r$, there are  natural isomorphisms
$$\Hom_R(H_1^\cts(\u), V)\cong\Hom_R(\Gr_r^WH_1^\cts(\u), V)\cong \left\{
  \begin{array}{lr}
    H_\cts^1(\G, V) & r<0\\
    0                       & r\geq 0
  \end{array}
\right.
$$
and a natural injection $\Hom_R(H_2^\cts(\u), V)\hookrightarrow H_\cts^2(\G, V)$ for $r\leq -2$, and\\
 $\Hom_R(H_2^\cts(\u), V)=0$ for $r>-2$, where $\u$ is the Lie algebra of $\U$.

\end{proposition}

\subsection{Application to Families of Curves} Suppose that $k$ is a field, that $T$ is a locally noetherian geometrically connected scheme over $k$, and that $C\to T$ is a curve of genus $g\geq 2$. Fix an algebraic closure $\bar k$ of $k$. Denote the base change to $\bar k$ of $C$ and $T$ by $C\otimes_k\bar k$ and $T\otimes_k\bar k$, respectively.  Let $\bar \eta:\Spec \Omega\to T\otimes_k\bar k$ be a geometric point of $T\otimes_k\bar k$.  By abuse of notation, $\bar\eta$ also denotes the image of $\bar \eta$ in $T$.  Denote the geometric fiber of $C\otimes_k\bar k$ over $\bar \eta$ by $C_{\bar\eta}$. Let $\bar x$ be a geometric point of the fiber $C_{\bar \eta}$. The images of $\bar x$ in $C\otimes_k\bar k$ and $C$ are also denoted by $\bar x$. Fix a prime number $\ell$ distinct from char$(k)$. In this section, $H_{\Zl}=\Het^1(C_{\bar \eta}, \Zl(1))$ and $H_{\Ql}=H_{\Zl}\otimes \Ql$. Let $R$ be the Zariski closure of the image of the natural monodromy representation
$$\rho_{T,\etabar}: \pi_1(T, \bar \eta)\to \GSp(H_{\Ql}).$$
Assuming that $R$ contains the homotheties\footnote{For instance, this is the case when $k$  is a number field.}, we have the central cocharacter defined by
$$\omega: \Gm\to R\hspace{.6in} z\mapsto z^{-1}\text{id}_H,$$
which we call the standard cocharacter\footnote{This definition is made this way so that weights from Hodge Theory and weighted completion agree on $H$.}.

\begin{lemma} The monodromy representation $\pi_1(C, \bar x)\to \GSp(H_{\Ql})$ factors through $\pi_1(T, \bar\eta)$.
\end{lemma}
\begin{proof}
This follows immediately from the existence of the commutative diagram
$$\xymatrix@R=10pt@C=10pt{
&\pi_1(C_{\bar\eta}, \bar x)\ar[r]\ar[d]&\pi_1(C, \bar x)\ar[r]\ar[d]&\pi_1(T, \bar \eta)\ar[r]\ar[d]&1\\
1\ar[r]&\Inn(\Pi^{(\ell)})\ar[r]&\Aut(\Pi^{(\ell)})\ar[r]&\Out(\Pi^{(\ell)})\ar[r]&1,
}
$$
where $\Pi^{(\ell)}$ denotes the maximal pro-$\ell$ quotient $\pi_1(C_{\bar\eta}, \bar x)^{(\ell)}$ of $\pi_1(C_{\bar\eta}, \bar x)$ and rows are exact. 
\end{proof}
Since the canonical map $\pi_1(C, \bar x)\to \pi_1(T, \bar \eta)$ is surjective, it follows that the monodromy representation $\pi_1(T, \bar x)\to R(\Ql)$ is also Zariski dense. Denote by $\cG_C$ and $\cG_T$ the weighted completions of $\pi_1(C, \bar x)$ and $\pi_1(T, \bar \eta)$ with respect to $\omega$ and their monodromy representations to $R$, respectively, and denote their prounipotent radicals by $\U_C$ and $\U_T$.
 Since the canonical map $\pi_1(C\otimes_k\bar k, \bar x)\to \pi_1(T\otimes_k\bar k, \bar \eta)$ is surjective, their images in $R(\Ql)$ are equal. Denote their common Zariski closure by $R^{\geom}$, which is a reductive subgroup of $R$. Denote by $\cG_C^\geom$ and $\cG_T^\geom$ the continuous relative completion of $\pi_1(\bar C, \bar x)$ and $\pi_1(\bar T, \bar \eta)$ with respect to their monodromy representations to $R^\geom(\Ql)$, respectively, and denote their prounipotent radicals by $\U_C^\geom$ and $\U^\geom_T$.
 By pushing out the exact sequence
$$\pi_1(C_{\bar\eta}, \bar x)\to \pi_1(C, \bar x)\to\pi_1(T,\bar \eta)\to1$$
along the surjection $\pi_1(C_{\bar\eta}, \bar x)\to\pi_1(C_{\bar\eta}, \bar x)^{(\ell)}$, we obtain the exact sequence
$$1\to\pi_1(C_{\bar\eta}, \bar x)^{(\ell)}\to\pi_1'(C, \bar x)\to\pi_1(T, \bar \eta)\to 1$$ that fits in the commutative diagram
$$
\xymatrix@R=10pt@C=10pt{
&\pi_1(C_{\bar\eta}, \bar x)\ar[r]\ar[d]&\pi_1(C, \bar x)\ar[r]\ar[d]&\pi_1(T, \bar \eta)\ar[r]\ar@{=}[d]&1\\
1\ar[r]&\pi_1(C_{\bar\eta}, \bar x)^{(\ell)}\ar[r]\ar[d]&\pi_1'(C, \bar x)\ar[r]\ar[d]&\pi_1(T, \bar \eta)\ar[r]\ar[d]&1\\
1\ar[r]&\Inn(\Pi^{(\ell)})\ar[r]&\Aut(\Pi^{(\ell)})\ar[r]&\Out(\Pi^{(\ell)})\ar[r]&1.
}
$$
  Denote by $\cG'_C$ the weighted completion of $\pi_1'(C, \bar x)$ with respect to $\omega$ and its monodromy representation $\pi_1'(C, \bar x)\to R(\Ql)$. 

\begin{lemma}With the notations above, there is a canonical isomorphism
$$\cG_C\cong \cG'_C.$$
Similarly, there is a canonical isomorphism
$$\cG_C^\geom\cong\cG_C'^\geom.$$
\end{lemma}
\begin{proof}
By the functoriality of weighted completion, there is a unique map $\phi: \cG_C\to\cG'_C$.
Denote the kernel of $\pi_1(C, \bar x)\to \pi_1'(C, \bar x)$ by $N$. Recall that $N$ is the kernel of the maximal pro-$\ell$ quotient $K\to K^{(\ell)}$, where $K$ is the kernel of the natural projection 
$\pi_1(C, \bar x)\to \pi_1(T, \bar \eta)$.  We have the commutative diagram:
$$\xymatrix@R=10pt@C=10pt{
1\ar[r]&N\ar[r]\ar[d]&\pi_1(C, \bar x)\ar[r]\ar[d]&\pi_1'(C,\bar x)\ar[r]\ar[d]&1\\
1\ar[r]&\U_C(\Ql)\ar[r]&\cG_C(\Ql)\ar[r]&R(\Ql)\ar[r]&1
}
$$
Since compact subgroups of $\U(\Ql)$ are pro-$\ell$ groups, the left vertical map must be trivial. Hence the canonical map $\pi_1(C, \bar x)\to \cG(\Ql)$ factors through $\pi_1'(C, \bar x)$. By the universal property of weighted completion, there exists a unique map $\psi:\cG'_C\to \cG_C$. It is easy to see that $\phi$ and $\psi$ are inverse to each other.
\end{proof}
Denote the continuous $\ell$-adic unipotent completion of $\pi_1(C_\etabar, \bar x)$ by $\cP$. It is a prounipotent $\Ql$-group. Since compact subgroups of $\Ql$-points of a prounipotent group is pro-$\ell$, the canonical map $\pi_1(C_\etabar, \bar x)\to \cP$ factors through $\pi_1(C_\etabar, \bar x)^{(\ell)}$, and furthermore, there is a unique isomorphism $\cP\cong \pi_1(C_\etabar, \bar x)^{(\ell),\un}_{/\Ql}$ of $\cP$ and the unipotent completion of the maximal pro-$\ell$ quotient of $\pi_1(C_\etabar, \bar x)$, since both completions admit the same universal property.
\begin{proposition}\label{prop.11}
With the notation as above:
\begin{enumerate}
\item  There are exact sequences
$$1\to\cP\to \cG_C\to \cG_T\to 1$$
and 
$$ 1\to \cP\to \cG^\geom_C\to \cG^\geom_T\to 1$$
of proalgebraic $\Ql$-groups such that the diagram
$$
\xymatrix@R=10pt@C=0pt{
1 \ar[rr] && \pi_1(C_\etabar)^{(\ell)} \ar[rr]\ar'[d][dd]\ar[dr] &&
\pi_1'(C\otimes_k\bar k,\bar x) \ar[rr]\ar'[d][dd]\ar[dr] &&
\pi_1(T\otimes_k\bar k,\bar \eta) \ar[rr]\ar'[d][dd]\ar[dr] && 1 \cr
& 1 \ar[rr] && \pi_1(C_\etabar)^{(\ell)} \ar[rr]\ar[dd] &&
\pi_1'(C,\bar x) \ar[rr]\ar[dd] &&
\pi_1(T,\etabar) \ar[rr]\ar[dd] && 1 \cr
1 \ar[rr] && \cP(\Ql) \ar'[r][rr]\ar[dr] &&
\cG_C^\geom(\Ql) \ar'[r][rr]\ar[dr] &&
\cG_T^\geom(\Ql) \ar'[r][rr]\ar[dr] && 1\cr
& 1 \ar[rr] && \cP(\Ql) \ar[rr] && \cG_C(\Ql) \ar[rr] &&
\cG_T(\Ql) \ar[rr] && 1
}
$$
commutes.
\item Every section $s$ of $\pi_1(C, \bar x)\to \pi_1(T, \bar \eta)$ induces sections $s^{(\ell)}$ and $\bar s^{(\ell)}$ of $\pi_1'(C, \bar x)\to \pi_1(T, \bar\eta)$ and $\pi_1'(C\otimes_k\bar k, \bar x))\to \pi_1(T\otimes_k\bar k, \etabar)$, respectively,  and sections $\sigma$ and $\sigma^\geom$ of $\cG_C\to \cG_T$ and $\cG_C^\geom\to \cG_T^\geom$, respectively,  such that the diagram
$$
\xymatrix@R=10pt@C=0pt{
\pi_1'(C\otimes_k\bar k,\bar x)\ar[dd]\ar[dr] &&
\pi_1(T\otimes_k\bar k,\etabar)\ar'[d][dd]\ar[ll]_{\bar s^{(\ell)}}\ar[dr]\cr
&\pi_1'(C,\bar x)\ar[dd] && \pi_1(T,\etabar)\ar[dd]\ar[ll]_(0.6){s^{(\ell)}}\cr
\cG_C^\geom(\Ql)\ar[dr] &&
\cG_T^\geom(\Ql) \ar[dr]\ar'[l]_(0.7){\sigma^\geom}[ll] \cr
&\cG_C(\Ql) && \cG_T(\Ql) \ar[ll]_\sigma
}
$$
commutes.
\end{enumerate}
\end{proposition}
\begin{proof}
The first assertion follows from the exactness criterion \cite[Prop. 6.11]{hain2}. The second follows from the universal property of weighted and relative completions.
\end{proof}

Denote the Lie algebras of $R$, $\cG_C$, $\cG_T$, $\U_C$, $\U_T$, $\cP$ by $\r$, $\g_C$, $\g_T$, $\u_C$, $\u_T$, $\p$, respectively. These admit natural weight filtrations as objects of the category of $\cG_C$-modules. By Proposition \ref{prop.3}, their $r$th graded quotient is an $R$-module of weight $r$. Since $H_1(\cP)=H_1(\p)$ is pure of weight $-1$, it follows that $\p\cong W_{-1}\p$, and the basic properties of weighted completion \cite[Prop. 3.4]{wei} implies that  we have
$$ \g_A=W_0\g_A,\,\,W_{-1}\g_A=\u_A,\,\,\text{and}\,\,\Gr_0^W\g_A\cong \r,$$
where $A=C$ and $A=T$.
The following corollary follows immediately from the fact that the functor $\Gr_\bullet^W$ is exact on the category of $\cG_C$-modules.
\begin{corollary} With the notation above:

\
There is an exact sequence
$$0\to \Gr_\bullet^W\p\to\Gr_\bullet^W\g_C\to\Gr_\bullet^W\g_T\to 0$$
of graded Lie algebras in the category of $R$-modules. 

\end{corollary}

\section{Weighted Completion of Arithmetic Mapping Class Groups}
In this section, we summarize and extend the results of \cite[\S 8]{hain2}. Suppose that $g$ and $n$ are integers satisfying $2g-2+n >0$. Fix prime numbers $p$ and $\ell\not=p$.  Denote the finite Galois cover of the moduli stack $\M_{g,n/\Z[1/\ell]}$ given by Propostion \ref{prop. 4} by $M^\lambda_{g,n}$. Fix a connected component of the base change to $ \Z_p^\ur$ of $M^\lambda_{g,n}$ and denote it by $M^\lambda_{\Z^\ur_p}$, where $\Z_p^\ur$ is the maximal unramified extension of $\Z_p$. For $R=\bar\Q_p$ and $R=\bar\F_p$, the base change $M^\lambda_{R}$ of $M^\lambda_{\Z_p^\ur}$ is a connected smooth variety over $R$.  Since $M^\lambda_{\Z_p^\ur}$ is of finite type over $\Z_p^\ur$, it is defined over some finite unramified extension $S$ of $\Z_p$. 
Denote the fraction field and residue field of $S$ by $L$ and $k$, respectively. Denote the absolute Galois group of $L$ and $k$ by $G_L$ and $G_k$, respectively.
Fix geometric points $\bar\eta$ of $M^\lambda_{\bar \Q_p}$ and $\bar\xi$ of $M^\lambda_{\bar\F_p}$. Let $C_{\bar y}$ be the fiber of the universal curve over $\bar y$, where $\bar y=\bar\eta$ and $\bar y=\bar\xi$. For a $\Zl$-module $A$, set
$$H_A:=H^1_\et(C_{\bar y}, A(1)).$$
Since the image of the $\ell$-adic cyclotomic character $\chi_\ell: G_{L}\to \Gm(\Zl)$ is infinite, the image of $\chi_\ell:G_{L}\to \Gm(\Ql)$ is Zariski dense. The image of the monodromy representation
$$\rho^\geom_{\bar\Q_p,\etabar}: \pi_1(M^\lambda_{\bar\Q_p}, \bar\eta)\to \Sp(H_{\Zl})$$
is of finite index in $\Sp(H_{\Zl})$, and hence it is Zariski dense in $\Sp(H_{\Ql})$.
The commutative diagram
$$\xymatrix{
1\ar[r]&\pi_1(M^\lambda_{\bar\Q_p}, \bar\eta)\ar[r]\ar[d]^{\rho^\geom_{\bar\Q_p}}&\pi_1(M^\lambda_{L}, \bar\eta)\ar[r]\ar[d]^{\rho_{L}}&G_{L}\ar[r]\ar[d]^{\chi_\ell}&1\\
1\ar[r]&\Sp(H_{\Ql})\ar[r]&\GSp(H_{\Ql})\ar[r]&\Gm(\Ql)\ar[r]&1
}
$$
implies that the image of the monodromy representation
$$\rho_{L,\etabar}:\pi_1(M^\lambda_{L}, \bar\eta)\to \GSp(H_{\Ql})$$
is also Zariski dense. Denote the weighted completion of $\pi_1(M^\lambda_{L}, \bar\eta)$ with respect to  $\rho_{L,\etabar}$ and the standard cocharacter $\omega$ by 
$$\cG_{M^\lambda_{L}}\,\,\text{and}\,\,\tilde\rho_{L, \etabar}: \pi_1(M^\lambda_{L}, \bar\eta)\to\cG_{M^\lambda_{L}}(\Ql).$$
Denote the pullback to $M^\lambda_{\Z^\ur_p}$ of the universal curve $\cC_{g,n}\to \M_{g,n}$ by $f:\cC^\lambda_{\Z^\ur_p}\to M^\lambda_{\Z^\ur_p}$. Let  $\pi: M^\lambda_{\Z^\ur_p}\to \Z^\ur_p$ be the structure morphism of $M^\lambda_{\Z^\ur_p}$ over $\Z^\ur_p$. 
\begin{proposition}\label{prop.5}
The image of the monodromy representation 
$$\rho^\geom_{\bar\F_p,\bar\xi}: \pi_1(M^\lambda_{\bar\F_p}, \bar\xi)\to\Sp(H_{\Zl})$$
is pro-$\ell$.
\end{proposition}
\begin{proof}
Since the kernel of the reduction map $\Sp(H_{\Zl})\to\Sp(H_{\Z/\ell^m\Z})$ is a pro-$\ell$ group, the statement then will follow, if the composition 
$$\rho^\geom_{\bar\F_p}: \pi_1(M^\lambda_{\bar\F_p}, \bar\xi)\overset{\rho^\geom}\to\Sp(H_{\Zl})\to\Sp(H_{\Z/\ell^m\Z})$$
is trivial. By the proper-smooth base change theorem \cite[Ch.6 \S4]{mil}, the sheaf $R^1f_{\ast} \mu_{\ell^m}$ is a constructible locally constant \'etale sheaf on $M^\lambda_{\Z^\ur_p}$. Its fiber over a geometric point $\bar y$ is isomorphic to $H^1_\et(C_{\bar y}, \mu_{\ell^m})=(\Z/\ell^m\Z)^{2g}$.  Denote $R^1f_{\ast}\mu_{\ell^m}$ by $\cF$. Let $\bar s_1$ be a geometric point lying over the generic point and $\bar s_2$ be the closed point of $\Z^\ur_p$. By a generalization of the proper-smooth base change theorem \cite[ SGA 1 Expos\'e XIII, 2.9]{sga1}, the specialization morphism, induced by the specialization $\bar s_1\to\bar s_2$,
$$H^0_\et(M^\lambda_{\bar\F_p}, \cF)=(\pi_\ast\cF)_{\bar s_2}\to(\pi_\ast\cF)_{\bar s_1}=H^0_\et(M^\lambda_{\bar\Q_p}, \cF)$$
is an isomorphism. Note that $H^0_\et(M^\lambda_{\bar\Q_p}, \cF)=(\cF_{\bar\eta})^{\pi_1(M^\lambda_{\bar\Q_p}, \bar\eta)}$. Since the standard representation $\G^\lambda_{g,n}\to \Sp(H_{\Zl})$ factors through the level $\ell^m$ subgroup $\G_{g,n}[\ell^m]$, the composition with the reduction mod-$\ell^m$ map
$$\G^\lambda_{g,n}\to \Sp(H_{\Z/\ell^m\Z})$$
is trivial, and so is the monodromy representation
$$\pi_1(M^\lambda_{\bar\Q_p}, \bar\eta)\to \Sp(H_{\Z/\ell^m\Z}).$$
Thus we have 
$$(\cF_{\bar\eta})^{\pi_1(M^\lambda_{\bar\Q_p}, \bar\eta)}=(\Z/\ell^m\Z)^{2g},$$
which implies that
$$(\cF_{\bar\xi})^{\pi_1(M^\lambda_{\bar\F_p}, \bar\xi)}=(\Z/\ell^m\Z)^{2g}.$$
Therefore, the monodromy $\rho^\geom:\pi_1(M^\lambda_{\bar\F_p}, \bar\xi)\to \Sp(H_{\Z/\ell^m\Z})$ is trivial.
\end{proof}
\begin{corollary}
The image of the monodromy representation
$$\rho^\geom_{\bar\F_p, \bar\xi}: \pi_1(\M_{g,n/\bar\F_p}[\ell^m], \bar\xi)\to \Sp(H_{\Zl})$$
is pro-$\ell$. 
\end{corollary}
\begin{proof}We use the same notation as in the proof of the above proposition.
Denote the automorphism group of the \'etale cover $M^\lambda_{\Z^\ur_p}\to \M_{g,n/\Z^\ur_p}[\ell^m]$ by $G$. Note that $H^0_\et(M^\lambda_{\bar\F_p}, \cF)^G=H^0_\et(\M_{g,n/\bar\F_p}[\ell^m], \cF)$ and that $G$ acts trivially on $H^0_\et(M^\lambda_{\bar\F_p},\cF)$ as it acts trivially on $H^0_\et(M^\lambda_{\bar\Q_p},\cF)$. Thus it follows that 
$$H^0_\et(\M_{g,n/\bar\F_p}[\ell^m], \cF)=(\Z/\ell^m\Z)^{2g},$$
which implies that the monodromy $\pi_1(\M_{g,n/\bar\F_p}[\ell], \bar\xi)\to \Sp(H_{\Zl})$ has a pro-$\ell$ image. 
\end{proof}
\begin{proposition}\label{prop.6} The image of the monodromy representation 
$$\rho^\geom_{\bar\F_p,\bar\xi}: \pi_1(M^\lambda_{\bar\F_p}, \bar\xi)\to\Sp(H_{\Zl})$$
has finite index in $\Sp(H_{\Zl})$. Consequently, the image of the monodromy 
representation 
$$\rho^\geom_{\bar\F_p, \bar\xi}: \pi_1(\M_{g,n/\bar\F_p}[\ell^m], \bar\xi)\to \Sp(H_{\Zl})$$
also has finite index.
\end{proposition}
\begin{proof}
 Consider the diagram
 $$\xymatrix@R=10pt{
 1\ar[r]&\pi_1(C_{\bar\xi}, \bar x')^{(\ell)}\ar[r]\ar@{=}[d]&\pi_1'(\cC^\lambda_{\bar\F_p}, \bar x')\ar[r]\ar[d]&\pi_1(M^\lambda_{\bar\F_p}, \bar\xi)\ar[r]\ar[d]&1\\
 1\ar[r]&\pi_1(C_{\bar\xi}, \bar x')^{(\ell)}\ar[r]\ar[d]&\pi_1'(\cC^\lambda_{\Z^\ur_p}, \bar x')\ar[r]\ar[d]^\phi&\pi_1(M^\lambda_{\Z^\ur_p}, \bar\xi)\ar[r]\ar[d]^{\phi'}&1\\
 1\ar[r]&\pi_1(C_{\bar\eta}, \bar x)^{(\ell)}\ar[r]&\pi_1'(\cC^\lambda_{\Z^\ur_p}, \bar x)\ar[r]&\pi_1(M^\lambda_{\Z^\ur_p}, \bar\eta)\ar[r]&1\\
 1\ar[r]&\pi_1(C_{\bar\eta}, \bar x)^{(\ell)}\ar[r]\ar@{=}[u]&\pi_1'(\cC^\lambda_{\bar\Q_p}, \bar x)\ar[r]\ar[u]&\pi_1(M^\lambda_{\bar\Q_p}, \bar\eta)\ar[r]\ar[u]&1,\\
 }
 $$
 whose rows are exact and the vertical maps between the second and third rows are isomorphisms. This diagram commutes once we fix an isomorphism $\phi: \pi_1(\cC^\lambda_{\Z^\ur_p}, \bar x')\cong \pi_1(\cC^\lambda_{\Z^\ur_p}, \bar x)$, which determines an isomorphism $\phi':\pi_1(M^\lambda_{\Z^\ur_p}, \bar\xi)\cong \pi_1(M^\lambda_{\Z^\ur_p}, \bar\eta)$. Fix such an isomorphism. The proof of Proposition \ref{prop.5} also shows that the monodromy representation $\rho_{\Z^\ur_p,\bar\xi}:\pi_1(M^\lambda_{\Z^\ur_p}, \bar\xi)\to \Sp(H_{\Zl})$ also has a pro-$\ell$ image, since $H^0_\et(M^\lambda_{\Z^\ur_p}, \cF)\cong H^0_\et(M^\lambda_{\bar\F_p}, \cF)$ by the generalization of proper-smooth base change theorem. Thus it follows that the image of the monodromy representation $\pi_1(\cC^\lambda_{\Z^\ur_p}, \bar x')\to\Sp(H_{\Zl})$ is also pro-$\ell$. This implies that the image of $\pi_1'(\cC^\lambda_{\Z^\ur_p}, \bar x')$ in $\Aut(\pi_1(C_{\bar\xi})^{(\ell)})$ under its conjugation action on $\pi_1(C_{\bar\xi}, \bar x')^{(\ell)}$ is also pro-$\ell$, and hence this conjugation action factors through $\pi_1(\cC^\lambda_{\Z^\ur_p}, \bar x')^{(\ell)}$. Since the center of $\pi_1(C_{\bar\xi}, \bar x')^{(\ell)}$ is trivial, it follows that the composition
 $$ \pi_1(C_{\bar\xi}, \bar x')^{(\ell)}\to \pi_1(\cC^\lambda_{\Z^\ur_p}, \bar x')^{(\ell)}\to \Aut(\pi_1(C_{\bar \xi})^{(\ell)})$$
 is injective. Thus by taking maximal pro-$\ell$ quotients of the above diagram, we obtain the commutative diagram
 $$\xymatrix@R=10pt{
 1\ar[r]&\pi_1(C_{\bar\xi}, \bar x')^{(\ell)}\ar[r]\ar@{=}[d]&\pi_1(\cC^\lambda_{\bar\F_p}, \bar x')^{(\ell)}\ar[r]\ar[d]&\pi_1(M^\lambda_{\bar\F_p}, \bar\xi)^{(\ell)}\ar[r]\ar[d]&1\\
 1\ar[r]&\pi_1(C_{\bar\xi}, \bar x')^{(\ell)}\ar[r]\ar[d]&\pi_1(\cC^\lambda_{\Z^\ur_p}, \bar x')^{(\ell)}\ar[r]\ar[d]&\pi_1(M^\lambda_{\Z^\ur_p}, \bar\xi)^{(\ell)}\ar[r]\ar[d]&1\\
 1\ar[r]&\pi_1(C_{\bar\eta}, \bar x)^{(\ell)}\ar[r]&\pi_1(\cC^\lambda_{\Z^\ur_p}, \bar x)^{(\ell)}\ar[r]&\pi_1(M^\lambda_{\Z^\ur_p}, \bar\eta)^{(\ell)}\ar[r]&1\\
 1\ar[r]&\pi_1(C_{\bar\eta}, \bar x)^{(\ell)}\ar[r]\ar@{=}[u]&\pi_1(\cC^\lambda_{\bar\Q_p}, \bar x)^{(\ell)}\ar[r]\ar[u]&\pi_1(M^\lambda_{\bar\Q_p}, \bar\eta)^{(\ell)}\ar[r]\ar[u]&1,\\
 }
 $$
 whose rows are exact and vertical maps are all isomorphisms. From this diagram, we see that the diagram
 $$\xymatrix@R=10pt{
 \pi_1(M^\lambda_{\bar\Q_p}, \bar\eta)\ar[d]&\pi_1(M^\lambda_{\bar\F_p}, \bar\xi)\ar[d]\\
  \pi_1(M^\lambda_{\bar\Q_p}, \bar\eta)^{(\ell)}\ar[r]^\cong\ar[d]&\pi_1(M^\lambda_{\bar\F_p}, \bar\xi)^{(\ell)}\ar[d]\\ 
  \Sp(H_{\Zl})\ar[r]^\cong&\Sp(H_{\Zl})\\
  }
 $$
 commutes, where the bottom isomorphism is induced by $\phi$. Since the composition of the two left-hand vertical maps is the standard representation $(\G^\lambda_{g,n})^\wedge\to \Sp(H_{\Zl})$, it has finite-index image,  and so does the monodromy representation $\rho^\geom_{\bar\F_p,\bar\xi}:\pi_1(M^\lambda_{\bar\F_p}, \bar\xi)\to \Sp(H_{\Zl})$. The density of the monodromy $\pi_1(\M_{g,n/\bar\F_p}[\ell^m], \bar\xi)\to \Sp(H_{\Zl})$ follows since its image in $\Sp(H_{\Zl})$ contains the image of
 $\pi_1(M^\lambda_{\bar\F_p}, \bar\xi)\to \Sp(H_{\Zl})$. 
\end{proof} 
By Proposition \ref{prop.6}, the image of $\rho^\geom_{\bar\F_p,\bar\xi}: \pi_1(M^\lambda_{\bar\F_p}, \bar\xi)\to \Sp(H_{\Ql})$ is Zariski dense. Since the image of the $\ell$-adic cyclotomic character $\chi: G_{k}\to\Z_\ell^\times$ is infinity, the image of $\chi_\ell: G_{k}\to\Gm(\Ql)$ is Zariski dense. The commutative digram
$$\xymatrix{
1\ar[r]&\pi_1(M^\lambda_{\bar\F_p}, \bar\xi)\ar[r]\ar[d]^{\rho^\geom_{\bar\F_p}}&\pi_1(M^\lambda_{k}, \bar\xi)\ar[r]\ar[d]^{\rho_{k}}&G_{k}\ar[r]\ar[d]^{\chi_\ell}&1\\
1\ar[r]&\Sp(H_{\Ql})\ar[r]&\GSp(H_{\Ql})\ar[r]&\Gm(\Ql)\ar[r]&1.
}
$$
implies that the monodromy representation 
$$\rho_{k,\bar\xi}: \pi_1(M^\lambda_{k}, \bar\xi)\to \GSp(H_{\Ql})$$
has Zariski dense image. Denote the weighted completion of $\pi_1(M^\lambda_{k}, \bar\xi)$ with respect to  $\rho_{k,\bar\xi}$ and the standard cocharacter $\omega$ by 
$$\cG_{M^\lambda_{k}}\,\,\,\,\,\text{and}\,\,\,\,\,\tilde\rho_{k,\bar\xi}: \pi_1(M^\lambda_{k}, \bar\eta)\to\cG_{M^\lambda_{k}}(\Ql).$$\\
\indent Let $\bar y$ denote $\bar\eta$ and $\bar\xi$.  Similarly, we have the weighted completion of $\pi_1(M^\lambda_S, \bar y)$ with respect to $\rho_{S, \bar y}: \pi_1(M^\lambda_S, \bar y)\to \GSp(H_{\Ql})$ and the cocharacter $\omega$, denoted by 
$$\cG_{M^\lambda_S},\,\,\,\,\,\text{and}\,\,\,\,\,\rho_{S, \bar y}: \pi_1(M^\lambda_S, \bar y)\to \cG_{M^\lambda_S}(\Ql).$$
Recall that $\cG^\geom_{g,n/\Ql}$  and $(\G_{g,n})^\wedge\to \cG^\geom_{g,n/\Ql}(\Ql)$ is the continuous relative completion of $(\G_{g,n})^\wedge$ with respect to the standard representation $(\G_{g,n})^\wedge\to\Sp(H_{\Ql})$. For $g\geq 3$, the continuous relative completion of $\pi_1(M^\lambda_{\bar\Q_p}, \bar\eta)$ with respect to its standard representation $\rho^\geom_{\bar\Q_p,\etabar}$ is isomorphic to $\cG^\geom_{g,n/\Ql}$ by Theorem \ref{thm.1}. Similarly, for $g\geq 3$, the continuous relative completion of $\pi_1(\M_{g,n/\bar\Q_p}[\ell^m], \etabar)$ with respect to its standard representation to $\Sp(H_{\Ql})$ is isomorphic to $\cG^\geom_{g,n/\bar\Q_p}$ \cite[Prop.~3.3]{hain0}. When the field $F$ is clear from  context, we will denote $\cG^\geom_{g,n/F}$ by $\cG^\geom_{g,n}$.

\begin{proposition} \label{prop.10}
The continuous relative completion of $\pi_1(M^\lambda_{\bar\F_p}, \bar\xi)$ with respect to $\rho^\geom_{\bar\F_p,\bar\xi}$ is isomorphic to $\cG^\geom_{g,n}.$ Similarly, the continuous relative completion of $\pi_1(\M_{g,n/\bar\F_p}[\ell^m], \bar\xi)$ with respect to $\rho^\geom_{\bar\F_p,\bar\xi}$ is isomorphic to $\cG^\geom_{g,n}.$ 
\end{proposition}

\begin{proof} Fix an isomorphism $\phi: \pi_1(M^\lambda_{\Z^\ur_p}, \bar\eta)\cong\pi_1(M^\lambda_{\Z^\ur_p}, \bar\xi)$. We have the following commutative diagram
 $$\xymatrix@R=10pt{
 \pi_1(M^\lambda_{\bar\Q_p}, \bar\eta)\ar[r]\ar[d]&\pi_1(M^\lambda_{\Z^\ur_p}, \bar\eta)\ar[r]^{\cong}\ar[d]&\pi_1(M^\lambda_{\Z^\ur_p}, \bar\xi)\ar[d]&\pi_1(M^\lambda_{\bar\F_p}, \bar\xi)\ar[l]\ar[d]\\
  \pi_1(M^\lambda_{\bar\Q_p}, \bar\eta)^{(\ell)}\ar[r]^{\cong}\ar[dr]&\pi_1(M^\lambda_{\Z^\ur_p}, \bar\eta)^{(\ell)}\ar[r]^{\cong}\ar[d]&\pi_1(M^\lambda_{\Z^\ur_p}, \bar\xi)^{(\ell)}\ar[d]&\pi_1(M^\lambda_{\bar\F_p}, \bar\xi)^{(\ell)}\ar[l]_{\cong}\ar[dl]\\
  &\Sp(H_{\Ql})\ar[r]^\cong&\Sp(H_{\Ql})&,\\
  }
 $$  where the isomorphism $\Sp(H_\Ql)\cong\Sp(H_\Ql)$ is induced by the isomorphism $\phi$ and the isomorphisms on the second row are ones in the proof of Theorem \ref{thm 1}. By taking the relative completion of each of the profinite groups with respect to its corresponding monodromy representation, we obtain the commutative diagram of proalgebraic $\Ql$-groups
 $$\xymatrix@R=10pt{
 \cG^\geom_{g,n}\ar[r]\ar[d]&\cG^\geom_{M^\lambda_{\Z^\ur_p}}\ar[r]^{\cong}\ar[d]&\cG^\geom_{M^\lambda_{\Z^\ur_p}}\ar[d]&\cG^\geom_{M^\lambda_{\bar\F_p}}\ar[l]\ar[d]\\
 \cG^{\geom,(\ell)}_{g,n}\ar[r]^{\cong}\ar[dr]&\cG^{\geom,(\ell)}_{M^\lambda_{\Z^\ur_p}}\ar[r]^{\cong}\ar[d]&\cG^{\geom,(\ell)}_{M^\lambda_{\Z^\ur_p}}\ar[d]&\cG^{\geom,(\ell)}_{M^\lambda_{\bar\F_p}}\ar[l]_{\cong}\ar[dl]\\
  &\Sp(H_{\Ql})\ar[r]^\cong&\Sp(H_{\Ql})&.\\
  }
 $$ 
 Since the vertical maps between the first and second rows are isomorphism by Proposition \ref{prop.7}, it follows that 
 $$ \cG^\geom_{M^\lambda_{\bar\F_p}}\cong\cG^\geom_{M^\lambda_{\Z^\ur_p}}\cong\cG^\geom_{g,n}.$$  A similar argument applies to the relative completion of $\pi_1(\M_{g,n/\bar\F_p}[\ell^m], \bar\xi)$. 
 
 \end{proof}
For a field $F$ whose $\ell$-adic cyclotomic character has an infinite image, denote by $\A_F$ the weighted completion of $G_F$ with respect to the $\ell$-adic cyclotomic character $\chi_\ell:G_F\to \Gm(\Ql)$ and $\omega: z\mapsto z^{-2}$. \\
\indent Throughout the rest of this section, for a prime $\ell$, let $M$ denote the \'etale covers $M^\lambda_{g,n}$ and $\M_{g,n}[\ell^m]$ of $\M_{g,n/\Z[1/\ell]}$. As in above, we fix a connected component of the base change to $S$ of $M$ and denote it by $M_S$, where $S$ is some finite unramified extension of $\Z_p$ over which $M$ decomposes as a finite disjoint union of geometrically connected components. Recall that $L$ and $k$ are the fraction field and the residue field of $S$, respectively.

\begin{proposition}[{\cite[8.1]{hain2}}]
Applying weighted completion to the two right-hand columns and relative completion to the left-hand column of diagram
$$\xymatrix{
1\ar[r]&\pi_1(M_{\bar\Q_p}, \bar\eta)\ar[r]\ar[d]^{\rho^\geom_{\bar\Q_p}}&\pi_1(M_{L}, \bar\eta)\ar[r]\ar[d]^{\rho_{L}}&G_{L}\ar[r]\ar[d]^{\chi_\ell}&1\\
1\ar[r]&\Sp(H_{\Ql})\ar[r]&\GSp(H_{\Ql})\ar[r]&\Gm(\Ql)\ar[r]&1
}
$$
gives a commutative diagram
$$\xymatrix{
&\cG^\geom_{g,n}\ar[r]\ar[d]&\cG_{M_{L}}\ar[r]\ar[d]&\A_{L}\ar[r]\ar[d]&1\\
1\ar[r]&\Sp(H_{\Ql})\ar[r]&\GSp(H_{\Ql})\ar[r]&\Gm(\Ql)\ar[r]&1
}
$$
whose rows are exact.  Similar results hold if we replace the sequence
$$1\to \pi_1(M_{\bar\Q_p}, \bar\eta)\to\pi_1(M_{L}, \bar\eta)\to G_{L}\to 1$$
with the exact sequence
$$
1\to \pi_1(M_{\bar\F_p}, \bar\xi)\to\pi_1(M_{k}, \bar\xi)\to G_{k}\to 1$$
and
$$1\to \pi_1(M_{\Z^\ur_p}, \bar y)\to\pi_1(M_{S}, \bar y)\to \pi_1(S, \bar y )\to 1,$$
where $\bar y=\bar \eta$ and $\bar y=\bar \xi$. 

\end{proposition}
Denote the prounipotent radicals of $\cG^\geom_{M_{\bar\Q_p}}$, $\cG^\geom_{M_{\Z^\ur_p}}$, and $\cG^\geom_{M_{\bar\F_p}}$ by $\U^\geom_{M_{\bar\Q_p}}$, $\U^\geom_{M_{\Z^\ur_p}}$, and $\U^\geom_{M_{\bar\F_p}}$, respectively. Denote the Lie algebras of $\cG^\geom_{M_{\bar\Q_p}}$, $\cG^\geom_{M_{\Z^\ur_p}}$, $\cG^\geom_{M_{\bar\F_p}}$, $\U^\geom_{M_{\bar\Q_p}}$, $\U^\geom_{M_{\Z^\ur_p}}$, and $\U^\geom_{M_{\bar\F_p}}$ by
$\g^\geom_{M_{\bar\Q_p}}$, $\g^\geom_{M_{\Z^\ur_p}}$, $\g^\geom_{M_{\bar\F_p}}$, $\u^\geom_{M_{\bar\Q_p}}$, $\u^\geom_{M_{\Z^\ur_p}}$, and $\u^\geom_{M_{\bar\F_p}}$, respectively.

\begin{proposition} Let $F=L$ and $k$ and $\bar y =\bar\eta$ and $\bar\xi$, respectively. If $2g-2+n>0$, then the natural action of $\pi_1(M_{F}, \bar y)$ on $\pi_1(M_{\bar F}, \bar y)$ induces an action of $\cG_{M_{F}}$ on $\g^\geom_{M_{\bar F}}$. Therefore, $\g^\geom_{M_{\bar F}}$ and $\u^\geom_{M_{\bar F}}$ are pro-objects of the category of $\cG_{M_{F}}$-modules, and thus admit natural weight filtrations. 
\end{proposition}
\begin{proof}
This follows from the facts that the induced action of $\pi_1(M_F, \bar y)$ on $\u^{\geom}_{M_{\bar F}}$ factors through $\GSp(H_\Ql)$ and that Johnson's work on the abelianization of the Torelli group \cite{joh} implies that $H_1(\u^\geom_{g,1})$ is of pure of weight $-1$. 
\end{proof}

\begin{remark}\label{rem.1}
The exactness of the functor $\Gr^W_\bullet$ and the fact that $H_1(\u^\geom_{M_{\bar F}})$ has weight $-1$ imply that $W_{-r}\u^\geom_{M_{\bar F}}$ is the $r$th term of the lower central series of $\u^\geom_{M_{\bar F}}$. This coincidence allows us to apply the results of \cite{hain0} in this paper.

\end{remark}
\begin{proposition} The isomorphisms
$$\g^\geom_{g,n}\cong\g^\geom_{M_{\bar\Q_p}}\cong\g^\geom_{M_{\bar\F_p}}$$
are morphisms in the category of $\cG_{M_{L}}$-modules. 
\end{proposition}

\begin{proof}
First consider the diagram
$$
\xymatrix@R=10pt@C=10pt{
1\ar[r]&\pi_1(M_{\bar\Q_p}, \bar \eta)\ar[r]\ar[d]&\pi_1(M_{L}, \bar \eta)\ar[r]\ar[d]&G_{L}\ar[r]\ar@{=}[d]&1\\
1\ar[r]&\pi_1(\M_{g,n/\bar\Q_p}, \bar\eta))\ar[r]&\pi_1(\M_{g,n/L}, \bar\eta)\ar[r]&G_{L}\ar[r]&1
,}
$$
whose rows are exact. $\pi_1(M_{L}, \bar\eta)$ acts on $\pi_1(\M_{g,n/\bar\Q_p}, \bar \eta)$ by conjugation via the homomorphism $\pi_1(M_L, \etabar)\to \pi_1(\M_{g,n/L}, \etabar)$. This conjugation action induces an action of $\cG_{M_{L}}$ on $\g^\geom_{g,n}$ and hence the isomorphism $\g^\geom_{M_{\bar\Q_p}}\to\g^\geom_{g,n}$ is a $\cG_{M_{L}}$-module homomorphism. Secondly, consider the diagram
$$\xymatrix@R=10pt@C=10pt{
1\ar[r]&\pi_1(M_{\bar\Q_p}, \bar \eta)\ar[r]\ar[d]&\pi_1(M_{L}, \bar \eta)\ar[r]\ar[d]&G_{L}\ar[r]\ar[d]&1\\
1\ar[r]&\pi_1(M_{\Z^\ur_p}, \bar \eta)\ar[r]\ar[d]&\pi_1(M_{S}, \bar \eta)\ar[r]\ar[d]&\pi_1(S, \bar\eta)\ar[r]\ar[d]&1\\
1\ar[r]&\pi_1(M_{\Z^\ur_p}, \bar \xi)\ar[r]&\pi_1(M_{S}, \bar \xi)\ar[r]&\pi_1(S, \bar\xi)\ar[r]&1\\
1\ar[r]&\pi_1(M_{\bar\F_p}, \bar \xi)\ar[r]\ar[u]&\pi_1(M_{k}, \bar \eta)\ar[r]\ar[u]&G_{k}\ar[r]\ar[u]&1.
}
$$
A choice of an isomorphism $\phi: \pi_1(M_{\Z^\ur_p}, \bar\eta)\cong\pi_1(M_{\Z^\ur_p}, \bar\xi)$ determines isomorphisms $\pi_1(S, \bar\eta)\cong\pi_1(S, \bar\xi)$ and $\pi_1(M_{S}, \bar\eta)\cong\pi_1(M_{S}, \bar\xi)$, which makes the above diagram commute. Pushing out this diagram along the surjection $\pi_1(M_{\Z^\ur_p}, \bar\xi) \to \pi_1(M_{\Z^\ur_p}, \bar \xi)^{(\ell)}$ induces the commutative diagram
$$\xymatrix@R=10pt{
1\ar[r]&\pi_1(M_{\bar\Q_p}, \bar \eta)^{(\ell)}\ar[r]\ar[d]&\pi_1'(M_{L}, \bar \eta)\ar[r]\ar[d]&G_{L}\ar[r]\ar[d]&1\\
1\ar[r]&\pi_1(M_{\Z^\ur_p}, \bar \eta)^{(\ell)}\ar[r]\ar[d]&\pi_1'(M_{S}, \bar \eta)\ar[r]\ar[d]&\pi_1(S, \bar\eta)\ar[r]\ar[d]&1\\
1\ar[r]&\pi_1(M_{\Z^\ur_p}, \bar \xi)^{(\ell)}\ar[r]&\pi_1'(M_{S}, \bar \xi)\ar[r]&\pi_1(S, \bar\xi)\ar[r]&1\\
1\ar[r]&\pi_1(M_{\bar\F_p}, \bar \xi)^{(\ell)}\ar[r]\ar[u]&\pi_1'(M_{k}, \bar \xi)\ar[r]\ar[u]&G_{k}\ar[r]\ar[u]&1,
}
$$
where rows are exact and all the left-hand vertical maps  and the vertical maps between the third and fourth rows are isomorphisms. Thus $\pi_1(M_{L}, \bar\eta)$ acts on $\pi_1(M_{\bar\F_p}, \bar\xi)^{(\ell)}$ through the conjugation action of $\pi_1'(M_{k}, \bar\xi)$ on $\pi_1(M_{\bar\F_p}, \bar\xi)^{(\ell)}$. Hence the induced isomorphism $\g^\geom_{M_{\bar\Q_p}}\cong\g^\geom_{M_{\bar\F_p}}$ is a $\cG_{M_{L}}$-module homomorphism. 
\end{proof}
Recall that for a prime number $\ell$, the corresponding finite \'etale cover $M^\lambda_{g,n}$ of $\M_{g,n}$ is defined over $\Z[1/\ell]$. 
Suppose that $F$ is a field of characteristic zero such that the image of the $\ell$-adic cyclotomic character $\chi_\ell: G_F\to \Gm(\Zl)$ is infinity and such that a connected component $M^\lambda_F$ of the base change to $F$ of $M^\lambda_{g,n}$ is geometrically connected. The weighted completion does not change for abelian levels; if $g\geq 3$, then for all $m\geq 1$ the natural homomorphism
$$\cG_{\M_{g,n/F}[m]}\to \cG_{\M_{g,n/F}}$$
is an isomorphism {\cite[Prop. 8.2]{hain2}.
\begin{proposition}
For all prime numbers $\ell\geq 3$ the natural homomorphisms
$$\cG_{M^\lambda_F}\to\cG_{\M_{g,n/F}[\ell^m]}\to \cG_{\M_{g,n/F}}$$
are isomorphisms.
\end{proposition}
\begin{proof} The same proof as in Proposition 8.2 \cite{hain2} with Proposition 6.6 \cite{Kahler} works.
\end{proof}

From this point, we will denote the weighted completions $\cG_{M^\lambda_F}$, $\cG_{\M_{g,n/F}[m]}$, and $\cG_{\M_{g,n/F}}$ by simply $\cG_{g,n/F}$ and omit $F$ when $F$ is clear from the context. Similarly, we will denote the Lie algebras $\g^\geom_{M_{\bar \Q_p}}$ and $\g^\geom_{M_{\bar\F_p}}$ by $\g^\geom_{g,n}$. They are pro-objects in the category of $\cG_{g,n}$-modules. 

\subsection{Variants} The comparison of the relative completions $\cG^\geom_{M_{\bar\Q_p}}$ and $\cG^\geom_{M_{\bar\F_p}}$ can be extended to the relative completion of the universal curve over $M$. 
Denote the pullback to $M_{\Z^\ur_p}$ of the universal curve $\cC_{g,n}$ by $\cC_{\Z^\ur_p}$. The diagram of profinite groups
$$
\xymatrix@R=10pt@C=10pt{
1\ar[r]&\pi_1(C_{\bar\xi}, \bar x')^{(\ell)}\ar[r]\ar@{=}[d]&\pi_1'(\cC_{\bar\F_p}, \bar x')\ar[r]\ar[d]&\pi_1(M_{\bar\F_p}, \bar\xi)\ar[r]\ar[d]&1\\
1\ar[r]&\pi_1(C_{\bar\xi}, \bar x')^{(\ell)}\ar[r]\ar[d]&\pi_1'(\cC_{\Z^\ur_p}, \bar x')\ar[r]\ar[d]&\pi_1(M_{\Z^\ur_p}, \bar\xi)\ar[r]\ar[d]&1\\
1\ar[r]&\pi_1(C_{\bar\eta}, \bar x)^{(\ell)}\ar[r]&\pi_1'(\cC_{\Z^\ur_p}, \bar x)\ar[r]&\pi_1(M_{\Z^\ur_p}, \bar\eta)\ar[r]&1\\
1\ar[r]&\pi_1(C_{\bar\eta}, \bar x)^{(\ell)}\ar[r]\ar@{=}[u]&\pi_1'(\cC_{\bar\Q_p}, \bar x)\ar[r]\ar[u]&\pi_1(M_{\bar\Q_p}, \bar\eta)\ar[r]\ar[u]&1\\
}
$$
commutes, where rows are exact,  after fixing an isomorphism\\ $\pi_1(\cC_{\Z^\ur_p}, \bar x')\cong\pi_1(\cC_{\Z^\ur_p}, \bar x)$, which determines isomorphisms $\pi_1(C_{\bar\xi}, \bar x')\cong\pi_1(C_{\bar\eta}, \bar x)$ and $\pi_1(M_{\Z^\ur_p}, \bar\xi)\cong\pi_1(M_{\Z^\ur_p}, \bar\eta)$. Applying continuous relative completion  to this diagram with respect to their natural monodromy representation to $\Sp(H_{\Ql})$ and taking Lie algebras, we obtain the commutative diagram 
$$
\xymatrix@R=10pt{
1\ar[r]&\p\ar[r]\ar@{=}[d]&\g^\geom_{\cC_{\bar\F_p}}\ar[r]\ar[d]&\g^\geom_{g,n}\ar[r]\ar[d]&1\\
1\ar[r]&\p\ar[r]\ar[d]&\g^\geom_{\cC_{\Z^\ur_p}}\ar[r]\ar[d]&\g^\geom_{g,n}\ar[r]\ar[d]&1\\
1\ar[r]&\p\ar[r]&\g^\geom_{\cC_{\Z^\ur_p}}\ar[r]&\g^\geom_{g,n}\ar[r]&1\\
1\ar[r]&\p\ar[r]\ar@{=}[u]&\g^\geom_{\cC_{\bar\Q_p}}\ar[r]\ar[u]&\g^\geom_{g,n}\ar[r]\ar[u]&1,\\
}
$$
where rows are exact and all the left and right-hand vertical maps are isomorphisms. Proposition \ref{prop.11} implies that the map $\p\to\g^\geom$ is injective, since the composition $\p\to\g^\geom\to\g$ is injective. Thus there is an isomorphism
$$\g^\geom_{\cC_{\bar\Q_p}}\cong \g^\geom_{\cC_{\bar\F_p}}.$$
As there is an isomorphism $\g^\geom_{M_{\bar\Q_p}}\cong\g^\geom_{g,n}$, there is an isomorphism $\g^\geom_{\cC_{\bar\Q_p}}\cong\g^\geom_{\cC_{g,n}}$, and these isomorphisms are morphisms in the category of $\cG_{\cC_{g,n}}$-modules. Hence we will denote the Lie algebras $\g^\geom_{\cC_{\bar\Q_p}}$ and $\g^\geom_{\cC_{\bar\F_p}}$ by $\g^\geom_{\cC_{g,n}}$. The canonical morphism $\cG_{\cC_{g,n}}\to\cG_{g,n}$ makes $\g^\geom_{g,n}$ a $\cG_{\cC_{g,n}}$-module. 
\begin{proposition}\label{prop.13}
Each section $x$ of the universal curve $f: \cC_{k}\to M_{k}$ induces a well-defined $\GSp(H_{\Ql})$-equivariant section of  $\Gr^W_\bullet f_\ast: \Gr^W_\bullet\g^\geom_{\cC_{g,n}}\to\Gr^W_\bullet\g^\geom_{g,n}$. 
\end{proposition}
\begin{proof}
By Proposition \ref{prop.11}, each section $x$ induces a section $\sigma^\geom$ of $f_\ast: \cG^\geom_{\cC_{\bar\F_p}}\to\cG^\geom_{M_{\bar\F_p}}$, which is well defined up to conjugation by an element of $\cP$. Thus the induced section $d\sigma_\ast$ of $df_\ast: \g^\geom_{\cC_{g,n}}\to\g^\geom_{g,n}$ is a morphism of $\cG_{\cC_{g,n}}$-modules and is well defined up to addition of a section of the form $\ad(u)\circ d\sigma^\geom$ with $u$ an element of $\p$. Since $\ad(u)\in W_{-1}\Der \g^\geom_{\cC_{g,n}}$, the sections $d\sigma^\geom$ and $d\sigma^\geom+\ad(u)\circ d\sigma^\geom$ induce the same section of $ \Gr^W_\bullet df_\ast: \Gr^W_\bullet \g^\geom_{\cC_{g,n}}\to\Gr^W_\bullet\g^\geom_{g,n}$. Denote this section by $\Gr^W_\bullet d\sigma^\geom$.  Since the action of $\U_{\cC_{g,n}}$ on $\g^\geom_{\cC_{g,n}}$ and $\g^\geom_{g,n}$ is negatively weighted, the graded Lie algebras $\Gr^W_\bullet \g^\geom_{\cC_{g,n}}$ and $\Gr^W_\bullet\g^\geom_{g,n}$ are $\GSp(H_{\Ql})$-modules and $\Gr^W_\bullet d\sigma^\geom$ is $\GSp(H_{\Ql})$-equivariant. 
\end{proof}

\section{The Characteristic Class of A Rational Point}In \cite{hain2}, Hain defined a characteristic class $\kappa_x$ for a $T$-rational point $x$ of the curve $C\to T$, where $T$ is a smooth variety over a field $k$ with char $(k)=0$. For our comparison purpose, we need to redefine this characteristic class for curves $C\to T$, where $T$ is defined over a more general base ring, e.g., $\Z_p$.  In this section, we will explain how this can be done and extend the results used in \cite{hain2} to positive characteristics. Let $B$ be a connected scheme. Suppose that $T$ is a geometrically connected smooth scheme over $B$ and that $f: C\to T$ is a curve of genus $g$. In this section, we associate a cohomology class $\kappa_x$ in $H_\et^1(T, R^1f_\ast \Ql(1))$ to a rational point $x\in C(T)$.

 Denote the relative Jacobian of $f:C\to T$ by $\pi: J_{C/T}\to T$. The scheme $J_{C/T}$ is a family of jacobians and is an abelian scheme over $T$.   Note that $J_{C/T}$ has a zero section $s_0: T\to J_{C/T}$. Let $\etabar: \Spec \Omega \to T$ be a geometric point of $T$. Denote the fiber of $f$ over $\etabar$ by $C_{\etabar}$ and the fiber of $J_{C/T}\to T$ over $\etabar$ by $(J_{C/T})_\etabar$. Let $\bar x$ be a geometric point of $C_{\etabar}$. Note that $(J_{C/T})_\etabar$ is the jacobian variety of the curve $C_\etabar$. When $\ell$ is not in char$(T)$, there are natural isomorphisms 
 $$\pi_1((J_{C/T})_\etabar, \bar x)^{(\ell)}\cong \pi_1(C_\etabar, \bar x)^{(\ell),\ab}\cong H_\et^1(C_\etabar, \Zl(1)),$$
 where $\ab$ denotes maximal abelian quotient.
 Denote the lisse sheaf $R^1f_\ast A(1)$ over $T$  by $\H_A$, where $A=\Zl$ or $\Ql$. Then we have  
 $$H_A:=H_\et^1(C_{\etabar}, A(1))= \big(\H_{A}\big)_\etabar.$$
  By \cite[SGA 1, Expos$\acute{\text{e}}$ XIII, 4.3, 4.4]{sga1} , there is an exact sequence
 $$
 1\to \pi_1((J_{C/T})_\etabar, \bar x)^{(\ell)}\to\pi_1'(J_{C/T}, \bar x)\to\pi_1(T, \etabar)\to 1.
 $$
Thus the zero section $s_0$ determines a splitting
$$\pi_1'(J_{C/T}, \bar x)\cong \pi_1((J_{C/T})_\etabar, \bar x)^{(\ell)}\rtimes \pi_1 (T, \etabar)\cong H_\Zl\rtimes \pi_1(T, \etabar),$$
which is well-defined up to conjugation action of $H_\Zl$. To each rational point $x\in C(T)$, we associate the divisor $D_x:= (2g-2)x-\omega_{C/T}$, where $\omega_{C/T}$ is the relative canonical divisor  of the family $C\to T$.  The divisor $D_x$ is homologically trivial on each geometric fiber, and hence gives a section of $J_{C/T}\to T$, which 
determines a class $\kappa_x$ in $$ H_\cts^1(\pi_1(T, \etabar), H_\Zl)\cong H^1_\et(T, \H_\Zl).$$
Tensoring with $\Ql$, we obtain a class in $H^1_\et(T, \H_\Ql)$, which we denote also by $\kappa_x$. 
\begin{remark}This class behaves well under base change.
\end{remark}

\subsection{Classes of the universal curve over $\M_{g,n}$}
Let $F$ be a field of characteristic zero. Suppose that $T$ is a noetherian geometrically connected scheme over $F$. 
 Denote the class in $H^1_\et(\M_{g,1/F}, \H_\Ql)$ of the tautological section of the universal curve $\cC_{g,1/F}\to \M_{g,1/F}$ by $\kappa$. 
This class is universal in the sense that for each rational point $x\in C(T)$, the class $\kappa_x\in H^1_\et(T, \H_\Ql)$ is the pullback of $\kappa$, i.e., $\kappa_x=\phi^\ast\kappa$, where $\phi: T\to \M_{g,1/F}$ is the morphism induced by $x$. Denote the class of the $j$th tautological section of the universal curve $\cC_{g,n/F}\to \M_{g,n/F}$ by $\kappa_j$. 
\begin{proposition}[{\cite[12.1]{hain2}}]\label{prop. 15}
If $g\geq 3$, $n\geq 0$, and $m\geq 1$, then for all fields $F$ of characteristic zero, 
$$H^1_\et(\M_{g,n/F}[m], \H_\Ql)=\Ql\kappa_1\oplus\Ql\kappa_2\oplus\cdots\oplus \Ql\kappa_n.$$ \qed
\end{proposition} 
Suppose that $p$ is a prime number, and that $\ell$ is a prime number distinct from $p$ and $m$ is a positive integer such that $\ell^m\geq 3$.  Denote a connected component of the base change to $\Z_p^\ur$ of $\M_{g,n}[\ell^m]$ by $\M_{\Z_p^\ur}[\ell^m]$. Denote the universal curve over $\M_{\Z_p^\ur}[\ell^m]$ by $\cC_{\Z_p^\ur}[\ell^m]$, and denote the relative Jacobian of $\cC_{\Z_p^\ur}[\ell^m]$ over $\M_{\Z_p^\ur}[\ell^m]$ by $J_{\Z_p^\ur}[\ell^m]$. For $A=\bar\Q_p$ and $\bar\F_p$, the base change to $A$ of $J_{\Z_p^\ur}[\ell^m]$  and $\M_{\Z_p^\ur}[\ell^m]$ are denoted by $J_{A}[\ell^m]$ and $\M_{A}[\ell^m]$, respectively. Let $\bar\xi$ and $\bar\eta$ be geometric points of $\M_{\bar\F_p}[\ell^m]$ and $\M_{\bar\Q_p}[\ell^m]$, respectively. We consider $\bar\xi$ and $\bar\eta$ as geometric points of $\M_{g,n/\Z_p^\ur}[\ell^m]$ via canonical morphisms induced by base change. 
Denote the fiber over $\bar\xi$ and $\etabar$ of $\cC_{\Z_p^\ur}[\ell^m]\to\M_{\Z_p^\ur}[\ell^m]$ by $C_{\bar\xi}$ and$C_{\etabar}$. Let $\bar x'$ and $\bar x$ be geometric points of $C_{\bar\xi}$ and $C_{\etabar}$, respectively. We have the diagram $(**)$
 $$\xymatrix@R=10pt@C=10pt{
 1\ar[r]&\pi_1(C_{\bar\xi}, \bar x')^{(\ell),\ab}\ar[r]\ar@{=}[d]&\pi_1'(J_{\bar \F_p}[\ell^m], \bar x')\ar[r]\ar[d]&\pi_1(\M_{\bar\F_p}[\ell^m], \bar\xi)\ar[r]\ar[d]&1\\
 1\ar[r]&\pi_1(C_{\bar\xi}, \bar x')^{(\ell),\ab}\ar[r]\ar[d]&\pi_1'(J_{\Z_p^\ur}[\ell^m], \bar x')\ar[r]\ar[d]&\pi_1(\M_{\Z_p^\ur}[\ell^m], \bar\xi)\ar[r]\ar[d]&1\\
 1\ar[r]&\pi_1(C_{\bar\eta}, \bar x)^{(\ell),\ab}\ar[r]&\pi_1'(J_{\Z_p^\ur}[\ell^m], \bar x)\ar[r]&\pi_1(\M_{\Z_p^\ur}[\ell^m], \bar\eta)\ar[r]&1\\
 1\ar[r]&\pi_1(C_{\bar\eta}, \bar x)^{(\ell),\ab}\ar[r]\ar@{=}[u]&\pi_1'(J_{\bar\Q_p}[\ell^m], \bar x)\ar[r]\ar[u]&\pi_1(\M_{\bar\Q_p}[\ell^m], \bar\eta)\ar[r]\ar[u]&1,\\
 }
 $$
that commutes after fixing an isomorphism $\pi_1(J_{\Z_p^\ur}[\ell^m], \bar x')\cong \pi_1(J_{\Z_p^\ur}[\ell^m], \bar x)$, which determines an isomorphism $\pi_1(\M_{\Z_p^\ur}[\ell^m], \bar\xi)\cong\pi_1(\M_{\Z_p^\ur}[\ell^m], \bar\eta)$. The rows of the diagram are exact and the vertical maps between the second and third row are isomorphisms. 

\begin{lemma} \label{lemma. 4}Suppose that $n\geq 1$. If $\bar \ast$ is a geometric point of $\M_{\Z_p^\ur}[\ell^m]$ and $\bar y$ is a geometric point of the fiber $C_{\bar \ast}$, then the sequence of the maximal pro-$\ell$ quotients
$$1\to \pi_1(C_{\bar\ast}, \bar y)^{(\ell),\ab}\to \pi_1(J_{\Z_p^\ur}[\ell^m], \bar y)^{(\ell)}\to\pi_1(\M_{\Z_p^\ur}[\ell^m], \bar \ast)^{(\ell)}\to 1 $$
of the exact sequence
$$1\to \pi_1(C_{\bar\ast}, \bar y)^{(\ell),\ab}\to \pi_1'(J_{\Z_p^\ur}[\ell^m], \bar y)\to\pi_1(\M_{\Z_p^\ur}[\ell^m], \bar \ast)\to 1 $$
is exact.
\end{lemma}

\begin{proof} A tautological section induces the closed immersion $ \psi: \cC_{\Z_p^\ur}[\ell^m]\to J_{\Z_p^\ur}[\ell^m]$ that 
makes the diagram 
$$\xymatrix@C=10pt{
(*)&1\ar[r]&\pi_1(C_{\bar\ast}, \bar y)^{(\ell)}\ar[r]\ar[d]&\pi_1'(\cC_{\Z_p^\ur}[\ell^m], \bar y)\ar[r]\ar[d]^{\psi_*}&\pi_1(\M_{\Z_p^\ur}[\ell^m], \bar\ast)\ar[r]\ar@{=}[d]&1\\
&1\ar[r]&\pi_1(C_{\bar \ast}, \bar y)^{(\ell),\ab}\ar[r]&\pi_1'(J_{\Z_p^\ur}[\ell^m], \bar y)\ar[r]&
\pi_1(\M_{\Z_p^\ur}[\ell^m], \bar\ast)\ar[r]&1
}
$$
commute, where the left-hand vertical map is the canonical projection.
Denote the kernel of the projection $\pi_1(C_{\bar\ast}, \bar y)^{(\ell)}\to\pi_1(C_{\bar\ast}, \bar y)^{(\ell),\ab} $ by $N$. Then $\psi_\ast$ induces an isomorphism 
$$\pi_1'(\cC_{\Z_p^\ur}[\ell^m], \bar y)/N\cong \pi_1'(J_{\Z_p^\ur}[\ell^m], \bar y).$$ 
Since by center-freeness $\pi_1(C_{\bar\ast}, \bar y)^{(\ell)}\to \pi_1(\cC_{\Z_p^\ur}[\ell^m], \bar y)^{(\ell)}$ is injective, there is an isomorphism 
$$\pi_1(\cC_{\Z_p^\ur}[\ell^m], \bar y)^{(\ell)}/N\cong \left(\pi_1'(\cC_{\Z_p^\ur}[\ell^m], \bar y)/N\right)^{(\ell)}.$$ Taking maximal pro-$\ell$ quotient of the diagram $(*)$ and pushing out along the surjection $\pi_1(C_{\bar\ast}, \bar y)^{(\ell)}\to\pi_1(C_{\bar\ast}, \bar y)^{(\ell),\ab}$, we obtain the commutative diagram
$$\xymatrix@C=10pt{
1\ar[r]&\pi_1(C_{\bar\ast}, \bar y)^{(\ell),\ab}\ar[r]\ar@{=}[d]&\pi_1(\cC_{\Z_p^\ur}[\ell^m], \bar y)^{(\ell)}/N\ar[r]\ar[d]&\pi_1(\M_{\Z_p^\ur}[\ell^m], \bar\ast)^{(\ell)}\ar[r]\ar@{=}[d]&1\\
&\pi_1(C_{\bar \ast}, \bar y)^{(\ell),\ab}\ar[r]&\pi_1(J_{\Z_p^\ur}[\ell^m], \bar y)^{(\ell)}\ar[r]&
\pi_1(\M_{\Z_p^\ur}[\ell^m], \bar\ast)^{(\ell)}\ar[r]&1,
}
$$
where the middle vertical map is an isomorphism. Thus it follows that the map $\pi_1(C_{\bar \ast}, \bar y)^{(\ell),\ab}\to\pi_1(J_{\Z_p^\ur}[\ell^m], \bar y)^{(\ell)}$ is injective. 
\end{proof}

\begin{proposition}Assume the notations above. If $g \geq 3$ and $n\geq 1$, then 
$$H^1_\et(\M_{g,n/\bar\F_p}[\ell^m], \H_\Ql)=\Ql\kappa_1\oplus\Ql\kappa_2\oplus\cdots\oplus \Ql\kappa_n.$$
Moreover, we have
$$H^1_\et(\M_{g,n/\F_q}[\ell^m], \H_\Ql)=\Ql\kappa_1\oplus\Ql\kappa_2\oplus\cdots\oplus \Ql\kappa_n,$$
where $\F_q=\F_p[\zeta_{\ell^m}]$ and $\zeta_{\ell^m}$ is a primitive $\ell^m$th root of unity. 

\end{proposition}

\begin{proof} By Lemma \ref{lemma. 4},  taking pro-$\ell$ completion of the diagram $(**)$ gives the commutative diagram
$$\xymatrix@R=10pt@C=10pt{
 1\ar[r]&\pi_1(C_{\bar\xi}, \bar x')^{(\ell),\ab}\ar[r]\ar@{=}[d]&\pi_1(J_{\bar \F_p}[\ell^m], \bar x')^{(\ell)}\ar[r]\ar[d]&\pi_1(\M_{\bar\F_p}[\ell^m], \bar\xi)^{(\ell)}\ar[r]\ar[d]&1\\
 1\ar[r]&\pi_1(C_{\bar\xi}, \bar x')^{(\ell),\ab}\ar[r]\ar[d]&\pi_1(J_{\Z_p^\ur}[\ell^m], \bar x')^{(\ell)}\ar[r]\ar[d]&\pi_1(\M_{\Z_p^\ur}[\ell^m], \bar\xi)^{(\ell)}\ar[r]\ar[d]&1\\
 1\ar[r]&\pi_1(C_{\bar\eta}, \bar x)^{(\ell),\ab}\ar[r]&\pi_1(J_{\Z_p^\ur}[\ell^m], \bar x)^{(\ell)}\ar[r]&\pi_1(\M_{\Z_p^\ur}[\ell^m], \bar\eta)^{(\ell)}\ar[r]&1\\
 1\ar[r]&\pi_1(C_{\bar\eta}, \bar x)^{(\ell),\ab}\ar[r]\ar@{=}[u]&\pi_1(J_{\bar\Q_p}[\ell^m], \bar x)^{(\ell)}\ar[r]\ar[u]&\pi_1(\M_{\bar\Q_p}[\ell^m], \bar\eta)^{(\ell)}\ar[r]\ar[u]&1,\\
 }
 $$
 whose rows are exact and the vertical maps between the second and third row are isomorphisms induced by change of base points.    Furthermore, the maps
 $$\pi_1(\M_{\bar\Q_p}[\ell^m], \etabar)^{(\ell)}\to\pi_1(\M_{\Z_p^\ur}[\ell^m], \etabar)^{(\ell)}\leftarrow\pi_1(\M_{\Z_p^\ur}[\ell^m], \bar\xi)^{(\ell)}\leftarrow \pi_1(\M_{\bar\F_p}[\ell^m], \etabar)^{(\ell)}$$
 are isomorphisms, and hence by exactness all the vertical maps are isomorphisms. This implies that there is an isomorphism 
 $$H_\cts^1\left(\pi_1(\M_{g,n/\bar\Q_p}[\ell^m],\etabar)^{(\ell)}, (\H_{\Zl})_{\bar\eta}\right)\cong
 H_\cts^1\left(\pi_1(\M_{g,n/\bar\F_p}[\ell^m],\bar\xi)^{(\ell)}, (\H_{\Zl})_{\bar\xi}\right).$$
 For $A=\bar\Q_p$, $\bar\F_p$,  $\bar \gamma=\bar\eta$, $\bar\xi$,  and $\bar y=\bar x$, $\bar x'$, respectively, the diagram
 $$\xymatrix@R=10pt@C=10pt{
  1\ar[r]&\pi_1(C_{\bar\gamma}, \bar y)^{(\ell),\ab}\ar[r]\ar@{=}[d]&\pi_1'(J_{A}[\ell^m], \bar y)\ar[r]\ar[d]&\pi_1(\M_{A}[\ell^m], \bar\gamma)\ar[r]\ar[d]&1\\
 1\ar[r]&\pi_1(C_{\bar\gamma}, \bar y)^{(\ell),\ab}\ar[r]&\pi_1(J_{A}[\ell^m], \bar y)^{(\ell)}\ar[r]&\pi_1(\M_{A}[\ell^m], \bar\gamma)^{(\ell)}\ar[r]&1}
  $$
is the pullback diagram along the surjection $\pi_1(\M_A[\ell^m], \bar\gamma)\to \pi_1(\M_A[\ell^m], \bar\gamma)^{(\ell)}$. Thus there is a canonical isomorphism
$$H_\cts^1\left(\pi_1(\M_{g,n/A}[\ell^m],\bar\gamma), (\H_{\Zl})_{\bar\gamma}\right)\cong
 H_\cts^1\left(\pi_1(\M_{g,n/A}[\ell^m],\bar\gamma)^{(\ell)}, (\H_{\Zl})_{\bar\gamma}\right).$$ Therefore, we have  isomorphisms
 \begin{align*}
 H_\et^1(\M_{g,n/\bar\Q_p}[\ell^m], \H_{\Zl})&\cong   H_\cts^1\left(\pi_1(\M_{g,n/\bar\Q_p}[\ell^m],\bar\eta), (\H_{\Zl})_{\bar\eta}\right)\\
 &\cong  H_\cts^1\left(\pi_1(\M_{g,n/\bar\F_p}[\ell^m],\bar\gamma), (\H_{\Zl})_{\bar\xi}\right)\\
 &\cong H_\et^1(\M_{g,n/\bar\F_p}[\ell^m], \H_{\Zl}) .
 \end{align*}
 Under this isomorphism, the classes $\kappa_j$ of the $j$th tautological section correspond in $H_\et^1(\M_{g,n/\bar\Q_p}[\ell^m], \H_{\Zl})$ and $H_\et^1(\M_{g,n/\bar\F_p}[\ell^m], \H_{\Zl})$. Hence our claim follows from Proposition \ref{prop. 15}. As to the second claim, the spectral sequence
 $$H^s(G_{\F_q}, H^t_\et(\M_{\bar\F_q}[\ell^m], \H_{\Zl}))\Rightarrow H^{s+t}_\et(\M_{\F_q}[\ell^m], \H_{\Zl})$$
 and the fact that $H^0_\et(\M_{\bar\F_q}[\ell^m], \H_{\Zl})=0$ imply that we have
 $$H^1_\et(\M_{\F_q}[\ell^m], \H_{\Zl})=H^1_\et(\M_{\bar\F_q}[\ell^m], \H_{\Zl})^{G_{\F_q}}\subset H^1_\et(\M_{\bar\F_q}[\ell^m], \H_{\Zl}).$$ 
 Since the tautological sections are defined over $\Z$ and hence defined over $\F_q$ by base change, the corresponding classes $\kappa_j$'s lie in $H^1_\et(\M_{\bar\F_q}[\ell^m], \H_{\Zl})^{G_{\F_q}}$. Tensoring with $\Ql$, we have 
 $$H^1_\et(\M_{\F_q}[\ell^m], \H_{\Ql})=H^1_\et(\M_{\bar\F_q}[\ell^m], \H_{\Ql})=\Ql\kappa_1\oplus\Ql\kappa_2\oplus\cdots\oplus \Ql\kappa_n.$$
\end{proof}
\subsection{The $\ell$-adic Abel-Jacobi map}
Suppose that $\pi: A\to T$ is an abelian scheme over a smooth scheme over a field $F$ whose fibers are polarized abelian varieties. For a prime number $\ell$ not equal to char$(F)$, the $\ell$-adic Abel-Jacobi map agrees with the association
$$A(T)\to H^1_\et(T, R^1\pi_\ast\Zl(1)) ,\,\,\,\,\,\,\,\,x\mapsto \kappa_x.$$

\begin{lemma}[{\cite[12.2]{hain2}}]
If $\pi:A\to T$ is a family of polarized  abelian varieties over a noetherian scheme $T$, then the kernel of the $\ell$-adic Abel-Jacobi map
$$A(T)\to H^1_\et(T,  \H_{\Zl})$$
 is the subgroup $\bigcap_n\ell^nA(T)$ of $\ell^\infty$-divisible points, where $\ell$ is not in char $(T)$. 
\end{lemma}

\begin{corollary}[{\cite[12.3]{hain2}}]\label{cor. 1}
With notations as above, if the group $A(T)$ of sections of $\pi:A\to T$ is finitely generated, then
the kernel of 
$$A(T)\to H^1_\et(T, \H_{\Ql})$$
is finite. 
\end{corollary}
\begin{remark}
By a generalization of the Mordell-Weil Theorem \cite{Ne} by N\'eron, when $T$ is a geometrically connected smooth variety over a field that is finitely generated over its prime subfield, $A(T)$ is finitely generated. This is the case, for example, for the universal curve $\cC_{g,n/\F_q}[\ell^m]\to \M_{g,n/\F_q}[\ell^m]$. 
\end{remark}
Applying this result to the relative Jacobian $\pi: J_{C/T}\to T$ associated to the family of curves $f: C\to T$, where $T$ is a geometrically connected smooth variety over a field $F$.
\begin{corollary}[{\cite[12.4]{hain2}}]\label{cor. 5}
 Assume that the group of sections $J_{C/T}(T)$ of $\pi: J_{C/T}\to T$ is finitely generated. If $x$ and $y$ are sections of $f: C\to T$ and $\kappa_x=\kappa_y$, then $x-y$ is torsion in $J_{C/T}(T)$. 
\end{corollary}

\subsection{ The image of $\kappa_j$ in $\Hom_{\GSp(H)}(H_1(\u^\geom_{g,n}), H)$}
Proposition \ref{prop. 16} implies that there is a natural isomorphism
$$H^1_\et(\M_{g,n/\bar\Q_p}[\ell^m], \H_{\Ql})\cong \Hom_{\GSp(H)}(H_1(\u^\geom_{g,n}), H_{\Ql}).$$ We can explicitly describe the image of the class $\kappa_x$ in $\Hom_{\GSp(H)}(H_1(\u^\geom_{g,n}), H_{\Ql})$. For $n\geq1$, define 
$$\Lambda^3_nH_\Ql:=\{u_1,\ldots,u_n)\in (\Lambda^3H_\Ql)^n:\bar{u}_1=\cdots=\bar{u}_n\}\otimes\Ql(-1),$$
where $\bar{u}_j$ is the image of $u_j$ in $\Lambda^3_0H:=\Lambda^3H/H$ \footnote{The representation $H_\Ql$ sits in $\Lambda^3H_\Ql$ via the inclusion $u\mapsto u\wedge\theta$, where $\theta$ is the polarization.} for each $j$.
Denote the $\GSp(H)$-equivariant projection $\Lambda^3_1H\to H$ by $h$. This projection is induced by twisting the projection $\Lambda^3H\to H(1):$
$$ x\wedge y\wedge z \mapsto \theta(x,y)z+\theta(y,z)x+\theta(z,x)y.$$
Denote the $\GSp(H)$-equivariant homomorphism $\Lambda^3_nH\to H$
$$\Lambda^3_n H\to \left(\Lambda^3_1H\right)^n\overset{pr_j}\to\Lambda^3_1H\overset{h}\to H$$
by $h_j$.

\begin{proposition}[{\cite[12.5 \& 12.6]{hain2}},{\cite[6.5]{hm1}}]
If $g\geq 3$ and $n\geq1$, for each $j=1,\ldots, n$, the $\GSp(H)$-equivariant homomorphism
$$H_1(\u^\geom_{g,n})\cong \Gr^W_{-1}\u^\geom_{g,n}\cong \Lambda^3_nH\overset{2h_j}\to H$$
corresponds to the class $\kappa_j$ under the isomorphism
$$H^1_\et(\M_{g,n/\bar\Q_p}[\ell^m], \H_{\Ql})\cong \Hom_{\GSp(H)}(H_1(\u^\geom_{g,n}), H).$$ 
\end{proposition}

Fixing an isomorphism $\pi_1(\M_{g,n/\Z_p^\ur}[\ell^m], \bar\eta)\cong \pi_1(\M_{g,n/\Z_p^\ur}[\ell^m], \bar\xi)$ determines the isomorphisms $(\H_{\Ql})_\etabar \cong (\H_{\Ql})_{\bar\xi}$ and $\u^\geom_{\M_{\bar\Q_p}[\ell^m]}\cong \u^\geom_{\M_{\bar\F_p}[\ell^m]}$ that 
make the diagram
$$
\xymatrix@R=10pt@C=10pt{H^1_\et\left(\M_{\bar\Q_p}[\ell^m],\,\H_{\Ql}\right)\ar[d]\ar[r]& \Hom_{\GSp(H)}\left(H_1\left(\u^\geom_{\M_{\bar\Q_p}[\ell^m]}\right),\,\,\,(\H_{\Ql})_\etabar\right)\ar[d]\\
H^1_\et\left(\M_{\bar\F_p}[\ell^m],\, \H_{\Ql}\right)\ar[r]& \Hom_{\GSp(H)}\left(H_1\left(\u^\geom_{\M_{\bar\F_p}[\ell^m]}\right),\,\,\, (\H_{\Ql})_{\bar\xi}\right)
}
$$
commute. Hence we have 
\begin{corollary}\label{cor. 17}
If $g\geq 3$ and $n\geq1$, for each $j=1,\ldots, n$, the $\GSp(H)$-equivariant homomorphism $2h_j$ corresponds to the class $\kappa_j$ under the isomorphism
$$H^1_\et(\M_{g,n/\bar\F_p}[\ell^m], \H_{\Ql})\cong \Hom_{\GSp(H)}(H_1(\u^\geom_{g,n}), H).$$ \qed
\end{corollary}

\begin{remark}
The $\GSp(H)$-equivariant projection 
$$\Lambda^3_nH=\Lambda^3_0H\oplus H_1\oplus\cdots\oplus H_n\to H_j$$
onto the $j$th copy of $H$ is equal to $h_j/(g-1)$ and corresponds to the class $\kappa_j/(2g-2)$ under this isomorphism.
\end{remark}

\section{Generic Sections of Fundamental Groups}\label{Generic Sections of Fundamental Groups}
The content of this section should be well known to experts. However, because of its key role in the proof of Theorem 2, we will give a brief introduction of the results needed in the proof. \\
\indent Suppose that $S$  is the spectrum of an excellent henselian discrete valuation ring $R$ whose residue field $k$ is a perfect field of characteristic $p\geq 0$.
Denote the fraction field of $R$ by $K$.  Fix an algebraic closure $\bar K$ of $K$. Suppose that $\pi: X\to S$ is a proper smooth morphism with geometrically connected fibers. 
Let $\bar x$ and $\bar x'$ be geometric points of the fibers $X_{\bar K}$ and $X_{\bar k}$, respectively.  We also consider $\bar x$ and $\bar x'$ as geometric points of $X$ via the morphisms $j: X_{\bar K}\to X$ and $i: X_{\bar k}\to X$ induced by base change. Fixing an isomorphism $\pi_1( X, \bar x)\cong \pi_1(X, \bar x')$ gives the commutative diagram $(\ast)$
$$\xymatrix@R=10pt@C=10pt{
1\ar[r]&\pi_1(X_{\bar K}, \bar x)\ar[r]\ar[d]^{\text{sp}}&\pi_1(X_K, \bar x)\ar[r]\ar[d]^{\text{sp}}& G_K\ar[r]\ar[d]&1\\
1\ar[r]&\pi_1(X_{\bar k}, \bar x')\ar[r]&\pi_1(X_k, \bar x')\ar[r]& G_k\ar[r]&1
}
$$
whose rows are exact and vertical maps are surjective. The surjective maps
$$\pi_1(X_{\bar K}, \bar x')\to\pi_1(X_{\bar k}, \bar x'),\,\,\,\,\,\,\,\,\pi_1(X_{ K}, \bar x)\to\pi_1(X_{ k}, \bar x')$$
 in the diagram are the specialization homomorphism defined in \cite[SGA 1, X]{sga1}. \\
\indent Denote the kernel of the natural map $G_K\to G_k$ by $I_k$. It is the Galois group of the maximal unramified subextension $K^\ur$ in $\bar K$ of $K$. For a section $s$ of $\pi_1(X_K, \bar x)\to G_K$, we define the {\it ramification} of $s$ to be the map
$$\ram_s=\text{sp}\circ s|_{I_k}: I_k\to \pi_1(X_{\bar k}, \bar x').$$ This sits in the commutative diagram
$$\xymatrix@R=10pt@C=10pt{
1\ar[r]&I_k\ar[r]\ar[d]^{ram_s}&G_K\ar[r]\ar[d]^{\text{sp}\circ s}& G_k\ar[r]\ar@{=}[d]&1\\
1\ar[r]&\pi_1(X_{\bar k}, \bar x')\ar[r]&\pi_1(X_k, \bar x')\ar[r]& G_k\ar[r]&1
.}
$$
From this, we see that $\ram_s^\ab: I_k^\ab\to \pi_1(X_{\bar k}, \bar x')^\ab$ is a $G_k$-equivariant map and that when $\ram_s$ is trivial, the section $s$ induces a section $s_0$ of $\pi_1(X_k, \bar x')\to G_k$. A section $s$ with trivial $\ram_s$ is called {\it unramified}. A section of $\pi_1(X_K)\to G_K$ induced by a rational point in $X_K(K)$ is unramified. \\
\indent Now, suppose that $\ell$ is a prime number distinct from char$(k)=p$. Pushing out the diagram $(\ast)$ along the surjection $\pi_1(X_{\bar k})\to \pi_1(X_{\bar k})^{(\ell)}$, we obtain the commutative diagram
$$\xymatrix@R=15pt@C=15pt{
1\ar[r]&\pi_1(X_{\bar K}, \bar x)^{(\ell)}\ar[r]\ar[d]^{\text{sp}^{(\ell)}}&\pi_1'(X_K, \bar x)\ar[r]\ar[d]^{\text{sp}'}& G_K\ar@/_.6pc/[l]_<(.1){s'}\ar[r]\ar[d]&1\\
1\ar[r]&\pi_1(X_{\bar k}, \bar x')^{(\ell)}\ar[r]&\pi_1'(X_k, \bar x')\ar[r]& G_k\ar[r]&1
.}
$$
The restriction of the composite $\text{sp}'\circ s'$ to $I_k$ induces the map
$$ \ram_s^{(\ell)}: \Zl(1)\to \pi_1(X_{\bar k})^{(\ell)}.$$

\begin{proposition}[{\cite[Prop.~91]{stix}}]\label{prop. unramified section}
With the same notation as in above, suppose that  the fibers of $\pi: X\to S$ are curves and that the residue field $k$ of $S$ is finitely generated over its prime subfield. Then $\ram_s(I_k)$ is a free pro-$p$ group. In particular, $\ram_s^{(\ell)}$ is trivial and each section of $\pi_1(X_k)\to G_K$ induces a section of $\pi_1'(X_{k})\to G_k$. 
\end{proposition}
Let $F$ be a finitely generated field. Suppose that $f:C\to T$ is a family of curves over  an irreducible regular scheme $T$ of finite type over a field $F$. Let $L$ be the function field of $T$ and $\ell$ a prime number distinct from char$(F)$.  Let $\bar \eta$ be a geometric generic point of $C$. The image of $\bar \eta$ in $T$ is a geometric generic point of $T$. In the following, fundamental groups are defined by using this choice of base points. Define the pro-$\ell$ sections of $\pi_1(C)\to\pi_1(T)$ to be the sections of $\pi_1'(C)\to\pi_1(T)$, where $\pi_1'(T)=\pi_1(T)/\ker(\Pi\rightarrow\Pi^{(\ell)})$ and $\Pi=\pi_1(C_\etabar)$.

\begin{corollary}\label{cor. 10}
Each section of $\pi_1(C_L)\to G_L$ induces a pro-$\ell$ section of $\pi_1(C)\to \pi_1(T)$. Consequently, there is a bijection between the set of conjugacy classes of pro-$\ell$ sections of $\pi_1(C_L)\to G_L$ and that of $\pi_1(C)\to \pi_1(T)$. 
\end{corollary}

\begin{proof}
 Proposition \ref{prop. unramified section} implies that each section of $\pi_1(C_L)\to G_L$ descends to a pro-$\ell$ section at each codimension-1 point of $T$ and Zariski-Nagata purity \cite[SGA 1 X Thm. 3.1]{sga1} then implies that it descends to a pro-$\ell$ section of $\pi_1(C)\to \pi_1(T)$. 

\end{proof}

\section{The Proof of Theorem 1 and 2}
Our proof of Theorem 1 is basically the same as the one given for Theorem 1 \cite{hain2} by Hain. His original proof of the theorem needed to be modified to work in positive characteristic. Our proof of Theorem 2 differs from Hain's proof of Theorem 2 \cite{hain2}; he studied the weighted completion of the fundamental group of the generic point of $\M_{g,n/\Q}$ using a density theorem \cite{hain} and non-abelian cohomology developed by Kim \cite{kim}. Our approach in this paper is to use the results in Section \ref{Generic Sections of Fundamental Groups}.
Recall that $p$ is a prime number, $\ell$ is a prime number distinct from $p$, and $m$ is a nonnegative integer. 
\begin{proposition}
\label{prop. 20}
Suppose that $g\geq 3$, $n\geq 1$, and $\ell^m\geq3$. If $x$ is a section of the universal curve $\cC_{g,n/\bar \F_p}[\ell^m]\to \M_{g,n/\bar \F_p}[\ell^m]$ and $\kappa_x=\kappa_j$, then $x$ is the $j$th tautological point $x_j$. 
\end{proposition}
\begin{proof}  
Without loss of generality, we may assume that $j=1$. The section $x$ is defined over some finite extension $\F_q$ of $\F_p$, which we may assume to contain a $\ell^m$th root of unity $\mu_{\ell^m}(\bar\F_p)$. Thus we consider $x$ as a section of $\cC_{g,n/\F_q}[\ell^m]\to\M_{g,n/\F_q}[\ell^m]$. Denote the relative Jacobian of $\cC_{g,n/\F_q}[\ell^m]\to \M_{g,n/\F_q}[\ell^m]$ by $J$. By Corollary \ref{cor. 5}, $t:=[x-x_1]$ is a torsion in $J(\M_{\F_q}[\ell^m])$. If $t=0$, then, since $g\geq 3$, we have $x=x_1$. If $t\not =0$ and $p^rt\not=0$ for any $r\geq1$, then the sections $x$ and $x_1$ are disjoint, since torsion points whose order is not divisible by $p$ are \'etale over the base. Hence they induce the morphism 
$$\M_{g,n/\F_q}[\ell^m]\to \M_{g,2/\F_q}[\ell^m]\,\,\,\,\,\,\,\,y\mapsto (C_y; x_1(y), x(y)),
$$
where $C_y$ is the fiber at $y$ of $\cC_{g,n/\F_q}[\ell^m]\to\M_{g,n/\F_q}[\ell^m]$. By Corollary \ref{cor. 17}, $\kappa_x=\kappa_1$ implies that the induced $\GSp(H)$-equivariant homomorphism
$$\phi: \Gr^W_{-1}\u^\geom_{g,n}=\Lambda^3_0H\oplus H_1\oplus\cdots\oplus H_n\to\Gr^W_{-1}\u^\geom_{g,2}=\Lambda^3_0H\oplus H_1\oplus H_2$$
is given by 
$$(v; u_1,\ldots,u_n)\mapsto(v;u_1,u_1).$$ This is impossible by Lemma 13.1 \cite{hain2}.
If $p^rt=0$ for some $r\geq1$, then $p^{r-1}t$ is a $p$-torsion in $J(\M_{\F_q}[\ell^m])$. Proposition 4.8 \cite{p-rank monodromy} implies that $p^{r-1}t=0$, so inductively we see that $t=0$. 
\end{proof}

\begin{proof}[Proof of Theorem 1] It is enough to show for the case $\ell^m\geq 3$. The valuative criterion of properness and the normality of $\M_{g,n/\bar\F_p}[\ell^m]$ implies that each $K$-rational point of $\cC_{g,n/\bar\F_p}[\ell^m]$ gives a unique  section of the universal curve. Hence we have $\cC_{g,n/\bar\F_p}[\ell^m](K)=\cC_{g,n/\bar\F_p}[\ell^m](\M_{g,n/\bar\F_p}[\ell^m])$.
Let $x$ be a section of $\cC_{g,n/\bar\F_p}[\ell^m]\to\M_{g,n/\bar\F_p}[\ell^m]$. By Corollary 10.3 \cite{hain2} the section $x$ induces a section $s_x$ of $\epsilon_n: \d_{g,n+1}\to \d_{g,n}$ (see \cite[\S 10]{hain2} for the definition of $\d_{g,n}$). By Proposition 10.8 \cite{hain2}, we have $s_x=s_j$ for some $j\in\{1,\ldots, n\}$. Recall that $s_j$ is the section of $\epsilon_n$ induced by the $j$th tautological point. Corollary \ref{cor. 17} implies that $\kappa_x=\kappa_j$ and thus we have $x=x_j$ by Proposition \ref{prop. 20}.
\end{proof}

\begin{proof}[Proof of Theorem 2]

Suppose that there is a section $s$ of $\pi_1(C, \bar x)\to G_L$. By Corollary \ref{cor. 10}, the section $s$ induces a pro-$\ell$ section $s^{(\ell)}$ of
$$1\to\pi_1(C_{\bar\eta}, \bar x)^{(\ell)}\to \pi_1'(\cC_{\F_q}[\ell^m], \bar x)\to \pi_1(\M_{\F_q}[\ell^m], \bar\eta)\to 1,$$ 
which induces a $\GSp(H)$-equivariant section of $\epsilon_0: \d_{g,1}\to\d_{g,0}$. By  Proposition 10.8 \cite{hain2}, there is no $\GSp(H)$-equivariant section of $\epsilon_0$. Therefore, there is no section of $\pi_1(C, \bar x)\to G_L$. 

\end{proof}

\end{document}